\documentclass{amsart}
\usepackage{mathrsfs}
\usepackage{stmaryrd,mathtools}
\usepackage{enumerate}
\usepackage{tikz-cd}
\usepackage[all]{xy}
\usepackage{aliascnt}

\usepackage{shuffle}
\usepackage{ytableau}

\usepackage[T1]{fontenc}

\usepackage[colorlinks=true, linkcolor=blue, citecolor=blue]{hyperref}
\usepackage{fullpage}
\usepackage{verbatim}
\usepackage{amssymb}
\usepackage{cleveref}
\usepackage[normalem]{ulem} % added by Claude, for \sout

\newtheorem{theorem}{Theorem}[section]

\newaliascnt{lemma}{theorem}
\newtheorem{lemma}[lemma]{Lemma}
\aliascntresetthe{lemma}

\newaliascnt{corollary}{theorem}
\newtheorem{corollary}[corollary]{Corollary}
\aliascntresetthe{corollary}

\newaliascnt{proposition}{theorem}
\newtheorem{proposition}[proposition]{Proposition}
\aliascntresetthe{proposition}

\newaliascnt{potato}{theorem}

\aliascntresetthe{potato}

\newaliascnt{definitionlemma}{theorem}

\aliascntresetthe{definitionlemma}

\newaliascnt{conjecture}{theorem}

\aliascntresetthe{conjecture}

\newaliascnt{question}{theorem}

\aliascntresetthe{question}

\theoremstyle{definition}

\newaliascnt{definition}{theorem}
\newtheorem{definition}[definition]{Definition}
\aliascntresetthe{definition}

\newaliascnt{remark}{theorem}
\newtheorem{remark}[remark]{Remark}
\aliascntresetthe{remark}

\newaliascnt{example}{theorem}
\newtheorem{examplex}[example]{Example}
\newenvironment{example}
  {\pushQED{\qed}\examplex}
  {\popQED\endexamplex}
\aliascntresetthe{example}

\newaliascnt{notation}{theorem}

\aliascntresetthe{notation}

\definecolor{darkblue}{rgb}{0.6,0,0.1}

\usepackage{tikz}
\usetikzlibrary{calc, shapes, backgrounds,arrows,positioning,plotmarks}
\tikzset{>=stealth',
  head/.style = {fill = white, text=black},
  plaque/.style = {draw, rectangle, minimum size = 10mm}, 
  pil/.style={->,thick},
  junct/.style = {draw,circle,inner sep=0.5pt,outer sep=0pt, fill=black}
  }

\newcommand{\thickslash}{\mathbin{\!\!\pmb{\fatslash}}}

\newcommand{\Z}{\mathbb{Z}}
\newcommand{\Q}{\mathbb{Q}}

\newcommand{\bG}{\mathbb{G}}
\newcommand{\bP}{\mathbb{P}}
\newcommand{\bA}{\mathbb{A}}

\newcommand{\bZ}{\mathbb{Z}}

\newcommand{\cG}{\mathcal{G}}
\newcommand{\cZ}{\mathcal{Z}}
\newcommand{\cX}{\mathcal{X}}

\newcommand{\cU}{\mathcal{U}}
\newcommand{\cY}{\mathcal{Y}}
\newcommand{\cF}{\mathcal{F}}
\newcommand{\cO}{\mathcal{O}}
\newcommand{\cB}{\mathcal{B}}
\newcommand{\cC}{\mathcal{C}}

\newcommand{\oH}{\operatorname{H}}
\newcommand{\cI}{\mathcal{I}}
\newcommand{\cL}{\mathcal{L}}

\newcommand{\cR}{\mathcal{R}}
\newcommand{\cS}{\mathcal{S}}
\newcommand{\cT}{\mathcal{T}}
\newcommand{\cA}{\mathcal{A}}
\newcommand{\cM}{\mathcal{M}}
\newcommand{\cW}{\mathcal{W}}

\newcommand{\bQ}{\mathbb{Q}}

\newcommand{\fppf}{\operatorname{fppf}}

\newcommand{\bmu}{\pmb{\mu}}

\DeclareMathOperator{\Out}{Out}
\DeclareMathOperator{\Aut}{Aut}

\DeclareMathOperator{\Isom}{Isom}

\DeclareMathOperator{\GL}{GL}

\DeclareMathOperator{\Hom}{Hom}

\DeclareMathOperator{\Spec}{Spec}

\DeclareMathOperator{\Tor}{Tor}
\DeclareMathOperator{\Frac}{Frac}
\DeclareMathOperator{\Band}{Band}

\newcommand{\spec}{\operatorname{Spec}}
\newcommand{\Gm}{\mathbb{G}_m}

\makeatletter
\@namedef{subjclassname@2020}{%
  \textup{2020} Mathematics Subject Classification}
\makeatother

\newif\ifhascomments \hascommentstrue
\ifhascomments
  \newcommand{\matt}[1]{{\color{red}[[\ensuremath{\spadesuit\spadesuit\spadesuit} #1]]}}
  \newcommand{\dori}[1]{{\color{blue}[[\ensuremath{\clubsuit\clubsuit\clubsuit} #1]]}}
  \newcommand{\elden}[1]{{\color{blue}[[\ensuremath{\clubsuit\clubsuit\clubsuit} #1]]}}
\else
  \newcommand{\elden}[1]{}
  \newcommand{\dori}[1]{}
  \newcommand{\matt}[1]{}
\fi

\begin{document}

\title{Root stack valuative criterion for good moduli spaces}

\author{Dori Bejleri}
\thanks{DB was partially supported by NSF grant DMS-2401483.}
\address[DB]{Department of Mathematics, University of Maryland, College Park, MD, 20742}
\email{dbejleri@umd.edu}

\author[Giovanni Inchiostro]{Giovanni Inchiostro}
\thanks{GI was partially supported by NSF grant DMS-2502104.}
	\address[GI]{Department of Mathematics, University of Washington, Seattle, Washington, USA}
	\email{ginchios@uw.edu}

\author{Matthew~Satriano}
\thanks{MS was partially supported by a Discovery Grant from the
  National Science and Engineering Research Council of Canada and a Mathematics Faculty Research Chair from the University of Waterloo.}
\address[MS]{Department of Pure Mathematics, University
  of Waterloo, Waterloo ON N2L3G1, Canada}
\email{msatrian@uwaterloo.ca}

\date{\today}
\subjclass[2020]{}

\begin{abstract} We prove a root stack valuative criterion for good moduli space maps and for gerbes for reductive groups under some mild assumptions on the residue characteristic. We give several applications to parahoric extension for torsors, rational points on stacks, gerbes and homogeneous spaces, and the geometry of fibrations.
\end{abstract}

\maketitle

\numberwithin{theorem}{section}
\numberwithin{lemma}{section}
\numberwithin{corollary}{section}
\numberwithin{proposition}{section}
\numberwithin{conjecture}{section}
\numberwithin{question}{section}
\numberwithin{remark}{section}
\numberwithin{definition}{section}
\numberwithin{example}{section}
\numberwithin{notation}{section}
\numberwithin{equation}{section}

\section{Introduction}

The goal of this paper is to prove the following root stack valuative criterion for good moduli spaces. 

\begin{theorem}\label{thm:main}
    Let $\cX\to X$ be a good moduli space where $\cX$ is an Artin stack with affine diagonal and of finite type over a locally Noetherian scheme $S$. Let $R$ be a DVR with fraction field $K$ and residue field $k$. Given a commutative diagram of solid arrows
    \[
    \xymatrix@R=4em@C=4em{
        &\spec K \ar@{^{(}-->}[dl]\ar[d]\ar[r] & \cX\ar[d]\\
        \sqrt[n]{\Spec R} \ar@{-->}[r] \ar@{-->}[rru]&\spec R\ar[r] & X
    }
    \]
\begin{comment}

    \[
    \xymatrix{
        \spec K\ar@{-->}[r]\ar[dr]\ar@/^1.2pc/[rr] & \cR\ar@{-->}[d]\ar@{-->}[r]\ar@{-->}[d] & \cX\ar[d]\\
        &\spec R\ar[r] & X
    }
    \]
\end{comment}
    there exists a root stack $\sqrt[n]{\spec R} \to\spec R$ and dotted arrows filling in the diagram. Moreover, if $\spec K\to \cX$ maps to the closed point of $\cX\times_X K$, then one can choose an extension such that also the closed point of $\sqrt[n]{\spec R}$ maps to the closed point of $\cX\times_X k$.
\end{theorem}

Good moduli spaces for algebraic stacks were introduced by Alper in his thesis \cite{Alper}, and are a generalization both of coarse moduli spaces for Deligne-Mumford stacks in characteristic $0$, and of morphisms, appearing in the context of GIT, of the form $[W(\cL)^{ss}/G]\to W/\!\!/_\cL G$, where $W$ is a projective variety with an action of a linearly reductive group $G$, $\cL$ is a $G$-linearized ample line bundle on $W$ and $W(\cL)^{ss}$ is its semistable locus. On a first approximation, these are morphisms $\cX\to X$ which are of the form $[\spec(A)/\GL_N]\to \spec(A^{\GL_N})$ \'etale locally on $X$ rather than Zariski locally, as in classical GIT \cite[Theorem 6.1]{AHRetalelocal}. Good moduli space maps if they exist are universal for maps to algebraic spaces and have very strong topological and cohomological properties. They are a fundamental tool in the study of algebraic stacks with positive dimensional stabilizers. 

The key difference between the root stack valuative criterion \Cref{thm:main} and the usual existence part of the valuative criterion for a good moduli space \cite[Theorem A.8]{AHLH} is that both the fraction field and the residue field of $R$ are preserved by the extension $\sqrt[n]{\spec R} \to \spec R$. This allows for many arithmetic applications (Section \ref{subsection:applications}) and for globalizing the valuative criterion from a DVR to a global curve $C$ (\Cref{prop:main_global} below). For Deligne-Mumford stacks over $\mathbb{C}$, the analogous valuative criterion has been used extensively in the theory of twisted stable maps \cite{AbramovichVistoli, ACV} and fibered surfaces \cite{MR3961332}. Recently it was generalized by Bresciani and Vistoli \cite{bresciani2024arithmetic} to tame stacks in any characteristic and appears in numerous applications including to heights of rational points on tame stacks \cite{ESZB, bejleri2024heightmodulicyclotomicstacks, DardaYasuda}, resolution of indeterminacy for rational maps to stacks \cite{mjjeon}, and classification of singular fibers in fibrations \cite{javier_hyperelliptic}. 

Beyond the realm of algebraic stacks with finite inertia, partial progress was made in \cite[Theorem 3.9]{di2022degenerations} and \cite[Theorem 1.2]{effective}. For the stack of $K$-semistable Fano varieties and of boundary polarized Calabi-Yau pairs, a proof of properness that shows one only needs to adjoin a root of the uniformizer of $R$ was given in \cite{properness_kmoduli_new} and \cite[Remark 7.2]{ABBDILW} respectively, though not necessarily that this can be done $\bmu_n$-equivariantly as in \Cref{thm:main}. See also \cite{warped} for related results in the context of motivic integration. \\

The fact that $\cX$ has a good moduli space is used to reduce to the case that $\cX$ is a gerbe for a smooth connected linearly reductive group scheme but in that case we prove the following result which in positive and mixed characteristic is more general than the setting of a good moduli space. 

\begin{remark}\label{rem:representable} We can factor the resulting morphism $\sqrt[n]{\spec R} \to \cX$ of \Cref{thm:main} through its relative coarse moduli space \cite[Theorem 3.1]{AOV} which is necessarily a root stack of possibly smaller degree \cite[Proposition 3.12]{bresciani2024arithmetic} so without loss of generality we can assume that this filling is representable. 
\end{remark}

\begin{theorem}\label{thm:main2} Let $R$ and $K$ be as above and let $\cX \to \sqrt[n]{\spec R}$ be a gerbe for a reductive group scheme $G \to \sqrt[n]{\spec R}$ over a root stack of $\spec R$. Suppose that the order of the Weyl groups of the fibers of $G\to \sqrt[n]{\spec R}$ is coprime to the residue characteristic. Then any $K$-point $x_K : \spec K \to \cX$ extends to a section $x : \mathcal{R} \to \cX$ where $\mathcal{R} \to \sqrt[n]{\spec R}$ is a further root stack.
\end{theorem}

The above theorem is the first step in generalizing \Cref{thm:main} to the case of \emph{adequate moduli spaces} \cite{Alperadequate}, at least under some assumptions. In characteristic $0$, adequate and good moduli spaces are equivalent but in positive and mixed characteristic, adequate moduli spaces are a larger class. For example, if $G$ is reductive and the residue characteristic is positive, then $BG \to S$ is an adequate moduli space but usually not a good moduli space. From \Cref{ex:tame_is_necessary} we can see that the following conditions on the stabilizer $G$ at a closed point seem to be necessary: the reduced connected  $(G_0)_{red} \subset G$ is as in \Cref{thm:main2} and the component group $G/(G_0)_{red}$ is a tame finite group scheme. Even if we impose these assumptions, there are two key places where the proof of \Cref{thm:main} uses the good moduli space rather than adequate moduli space assumption: $1)$ the application of canonical reduction of stabilizers \cite{ER} in \Cref{thm_reduction_to_gerbe_general_case} and in particular the fact that the reduced connected component of the identity behaves well in families, and $2)$ the replacement for Kempf's Theorem when $K$ is not perfect in \Cref{lemma_kempf_for_lin_reductive}. 

However, in mixed characteristic the root stack valuative criterion still holds \emph{for all but finitely many residue characteristics.}
This is made precise by the following theorem. 

\begin{theorem}\label{thm_mixed_char_adequate}Let $\cX$ be a finite type stack with affine diagonal over $\spec \Z$ and $\pi \colon \cX \to X$ a morphism to an algebraic space which is a good moduli space over $\spec \Q$ (e.g. if $\pi$ is an adequate moduli space). Then there exists some $m > 0$ such that the root stack valuative criterion holds for the basechange of $\pi$ to $\Z[1/m]$ with respect to mixed characteristic DVRs. That is, if $R$ is a DVR with characteristic zero fraction field $K$ and residue characteristic $p \nmid m$, and we are given a commutative diagram of solid arrows as below,
    \[
    \xymatrix@R=4em@C=4em{
        &\spec K \ar@{^{(}-->}[dl]\ar[d]\ar[r] & \cX\ar[d]\\
        \sqrt[n]{\Spec R} \ar@{-->}[r] \ar@{-->}[rru]&\spec R\ar[r] & X
    }
    \]
    there exists a root stack $\sqrt[n]{\spec R} \to\spec R$ and dotted arrows filling in the diagram. 
\end{theorem}

\begin{remark}\label{rem:spreading}
    If $\cX \to X$ is a good moduli space of a finite type stack with affine diagonal over an algebraically closed field of characteristic $0$, we may write $\cX \to X$ as the pullback of $\cX_0 \to X_0$ over some finitely generated field $K/\Q$. Then we may spread out $\cX_0 \to X_0$ to some morphism $\cX' \to X'$ over $S$ where $S$ if a finite type $\Z$-scheme with fraction field $K$. Now \Cref{thm_case_of_gerbe_over_root_stack} implies that the root stack valuative criterion holds for $\cX' \to X'$ with respect to DVRs of mixed characteristic $(0,p)$ for all but finitely many $p$. In particular, this may be applied in the setting of $p$-adic integration on $\cX' \to X'$ for $p \gg 0$ large enough (\Cref{sec:padic}). 
\end{remark}

Next we globalize the valuative criterion to show that rational maps from a global curve to an Artin stack with a proper good moduli space can be extended after a root stack. \Cref{prop:main_global} yields the existence of \emph{tuning stacks} for $K$-points of $\cX$ in the sense of \cite{ESZB} and is the first step in developing the theory of heights on stacks for Artin stacks with a good moduli space. 

\begin{proposition}\label{prop:main_global}
    Suppose $C$ is a locally excellent, regular, 1-dimensional scheme (e.g. a smooth curve over a field, or a Dedekind domain flat over $\bZ$). Let $\cX$ be a stack as in Theorem \ref{thm:main} with proper good moduli space or a gerbe over a root stack of $C$ as in \Cref{thm:main2}. Then any rational map $\varphi : C \dashrightarrow \cX$ extends to a morphism $f\colon \cC \to \cX$ from some root stack $\cC \to C$ along the complement of a domain of definition for $\varphi$. 
\end{proposition}

\begin{corollary}\label{cor_global} Let $C$ be as in Proposition \ref{prop:main_global} and let $K = k(C)$ be its fraction field. Let $\cX$ be a stack as in Theorem \ref{thm:main} with a proper good moduli space or a gerbe over a root stack of $C$ as in \Cref{thm:main2}. Then any $K$-point $x_K : \spec K \to \cX$ extends to a morphism $\cC \to \cX$ from some root stack of $C$.     
\end{corollary}

There is also a version of number fields using \Cref{thm_mixed_char_adequate}. 

\begin{corollary}\label{cor:number_fields} Let $\cX \to X$ be as in \Cref{thm_mixed_char_adequate} and suppose that $X$ is proper over $\Z[\frac{1}{m}]$ for some $m$. Then up to increasing $m$, for any number field $K$ and any $K$-point $x_K : \spec K \to \cX$, there exists a root stack $\cC \to \spec \cO_K[\frac{1}{m}]$ and an extension of $x_K$ to a morphism $\cC \to \cX$.   
\end{corollary}

Finally, Theorem \ref{thm:main} combined with the existence criterion of \cite{AHLH} immediately implies the following relative version. We refer the reader to \emph{loc.cit.} for the definitions of $\Theta$-reductive and $S$-complete morphisms.  

\begin{corollary}
    Let $\pi : \cX \to \cY$ be a finite type morphism of Artin stacks. Suppose that \begin{enumerate}[(i)]
    \item $\pi$ has affine diagonal, 
    \item the relative inertia groups at all closed points of the fibers of $\pi$ are linearly reductive,
    \item $\pi$ is $\Theta$-reductive,
    \item $\pi$ is $S$-complete, and
    \item $\pi$ satisfies the existence part of the valuative criterion of properness. 
    \end{enumerate} 
    Then for any commutative diagram of solid arrows below where $R$ is a DVR with fraction field $K$
    \[
    \xymatrix@R=4em@C=4em{
        &\spec K \ar@{^{(}-->}[dl]\ar[d]\ar[r] & \cX\ar[d]\\
        \sqrt[n]{\Spec R} \ar@{-->}[r] \ar@{-->}[rru]&\spec R\ar[r] & \cY
    }
    \]
    there exists a root stack $\sqrt[n]{\spec R}$ and dotted arrows filling in the diagram.     
\end{corollary}

\begin{remark}
    Note that assumption $(ii)$ follows from the other assumptions in characteristic $0$. 
\end{remark}

\subsection{Applications}\label{subsection:applications}Throughout this section, $R$ is a DVR with fraction field $K$ and residue field $k$. 

\subsubsection{Torsors and parahoric extension} Consider the special case of \Cref{thm:main2} where $\cX = \cB G$ is the classifying stack of a reductive group over a root stack of $R$. We say that a reductive group has tame Weyl group if its Weyl group has order coprime to the residue characteristics (Definition \ref{def_tame}). Note that some condition on the group is necessary as illustrated by Example \ref{ex:tame_is_necessary}. Then Theorem \ref{thm:main2} implies that every $G$-torsor over $K$ is pulled back from some $G$-torsor over some root stack of $\spec R$. In fact, if $G$ is defined over $\spec R$, the order of the root stack necessary is also uniformly bounded in terms of $G$. 

\begin{corollary}\label{cor:torsors}
    Let $G \to$ {$\spec R$} be a group, with $G$ reductive and with tame Weyl group. Then there exists an integer $n$ depending only on $G$ such that the natural map 
    $$
    \underline{\Hom}(\textcolor{black}{\sqrt[n]{\spec R}}, \cB G) \to \underline{\oH}^1(K,G)
    $$
    is essentially surjective where $\underline{\Hom}$ denotes the groupoid of maps and $\underline{\oH}^1$ denotes the groupoid of torsors.
\end{corollary}

Note that if $R$ is strictly henselian, the groupoid of $G$-torsors on $\sqrt[n]{\spec R}$ is equivalent to that of $G$-torsors over $\cB \bmu_{n,k}$. Thus we get an essentially surjective functor 
$$
\underline{\Hom}(\cB \bmu_{n,k}, \cB G) = \left[\Hom(\bmu_{n,k}, G)/G\right] \to \underline{\oH}^1(K,G)
$$
where the action of $G$ is by conjugation. By \cite[Lemma 7.1]{martens2015}, all such homomorphisms come from co-characters of $G$ over the residue field $k$. 

%\begin{remark}
%Compare \Cref{cor:torsors} with \cite[Lemma 3.5(d)]{ED_specialization}.
%\end{remark}
%\begin{remark} 
%The kernel of the natural map in \Cref{cor:torsors} should be related to the $\frac{1}{n}\mathbb{Z}$-points of the Bruhat-Tits building of $G$ over $K$ (see \cite[Theorem 3.17]{groechenig2025twistedpointsquotientstacks}).
%\end{remark}

One of the main interests in \Cref{prop_case_BG} is that it can be interpreted as a parahoric extension theorem for $G$-bundles on $K$. Indeed by work of many authors \cite{balaji_seshadri, pappas_rapoport, MR4703674, gs_parahoric} \cite[Section 5]{sheng2024nonabelianhodgecorrespondenceprincipal}, a reductive group over a root stack $\cR \to \spec R$ is equivalent data to a parahoric model $\cG \to \spec R$ of $G_K$ and the category of parahoric $\cG$-torsors over $R$ is equivalent to the category of $G$-torsors on $\cR$. Thus combining \Cref{prop_case_BG} with \Cref{prop:main_global} and \Cref{cor_global}, we get the following parahoric extension result. 

\begin{proposition}\label{prop_parahoric}
    Suppose $G$ is as in \Cref{thm:main2}. Then any $G$-torsor over $K$ extends to a parahoric torsor over $R$. The same holds if $R$ and $K$ are replaced by a global curve as in \Cref{prop:main_global} and its function field respectively. 
\end{proposition}

\begin{example}\label{ex:conic_tame}
Suppose $R$  has residue characteristic $\neq 2$ and let $C \subset \mathbb{P}^2_K$ be a smooth conic which corresponds to a map $\spec K \to \cB PGL_2$. Let $\cC \to \spec R$ be the closure of $C$ in $\bP^2_R$. Then the central fiber is either a smooth conic, a doubled line, or a union of two lines. In the first case, the $PGL_2$-torsor extends to $R$. In the second case, the normalization of the pullback of $\cC$ to $\sqrt[2]{\spec R}$ has good reduction and thus the torsor extends to the second root stack. In the third case, the total space has an $A_n$ singularity at the node of the central fiber for some $n \geq 0$ (where by convention an $A_0$ singularity is a regular point). If $n = 0$, blowing up the node and contracting the strict transform of the central fiber yields a model with good reduction so we are in the first case. If $n = 1$, blowing up the node and contracting the strict transform of the central fiber yields a model with a double line so we are in the second case. For $n \geq 2$, blowing up the node and contracting the strict transform of the central fiber yields a new central fiber which is still a union of two lines but now has an $A_{n-2}$ singularity. Repeating this operation puts us back in either the first or the second case depending on the parity of $n$. Thus for $\cB \mathrm{PGL}_2$ in odd characteristic, any torsor over $K$ extends up to a degree $2$ root stack. 
\end{example}

\begin{remark}
    In \cite{bejleri_jeon} \Cref{prop_parahoric} is generalized to prove that a $G$-torsor on $U$ extends to a parahoric torsor on some normal crossings compactification $X \supset U$ when $U$ is a smooth of any dimension in characteristic $0$ or regular of dimension $2$ in positive and mixed characteristic. 
\end{remark}

\subsubsection{The Lang-Nishimura Theorem for Artin stacks} Following \cite{bresciani2024arithmetic}, we give a generalization of Lang-Nishimura to Artin stacks with a good moduli space. We expect that this will have applications to questions on essential dimension, fields of definition and fields of moduli for Artin stacks as in e.g. \cite{BV2,ed_stacks}.

\begin{proposition}[Lang–Nishimura theorem]\label{prop_LN}
    Let $S$ be a scheme and $\cX\dashrightarrow \cY$ a rational map of algebraic stacks
over $S$, with $\cX$ locally noetherian and integral and $\cY$ admitting a good moduli space $\cY\to Y$ which is proper over $S$. Let
$k$ be a field, $s: \spec k \to S$ a morphism. Assume that $s$ lifts to a regular point
$\spec k\to \cX$; then it also lifts to a morphism $\spec k\to \cY$.
\end{proposition} 

The following is a straightforward corollary. 

\begin{corollary}\label{cor:LN_1}
     Suppose that a linearly reductive and special group $G$ acts on an affine scheme $\spec A$ and that the ring of invariants $A^G$ is an $R$-algebra. Assume also that $\spec A\otimes_R K$ has a $K$-point, and the composition $\spec K\to\spec A\to\spec A^G$ extends to $\spec R\to \spec A^G$. Then $\spec A \otimes_R k$ has a $k$-point.
\end{corollary}

\subsubsection{Gerbes and homogeneous spaces} Next we give some applications of Theorem \ref{thm:main2} to gerbes and homogeneous spaces over DVRs which may be viewed as a version of Grothendieck-Serre for gerbes and non-principal homogeneous spaces. 

\begin{proposition}\label{prop_gs_tame}
    Let $G \to \spec R$ be a reductive group. Then any class $\cG \in \oH^2(R,G)$ which is trivial over $\Spec K$ becomes trivial after passing to a root stack. In particular, if $\cG|_K = 0$ then $\cG|_k = 0$. 
\end{proposition}

As a special case, we can consider a homogeneous space $V$ for a group $H$ which yields a gerbe $[V/H]$ for the stabilizer group.

\begin{corollary}\label{cor:homogeneous_1} Let $H \to \spec R$ be a special affine group scheme. Let $V \to \spec  R$ be an affine homogeneous space for $H$ and suppose that the geometric stabilizer groups are reductive groups with tame Weyl group. If $V(K) \neq \emptyset$ then $V(k) \neq\emptyset$. 
\end{corollary}

\noindent As unirationality for homogeneous spaces is equivalent to admitting a rational point, we obtain the following. 

\begin{corollary}\label{cor:unirationality} Unirationality specializes along DVRs for homogeneous spaces for special reductive groups with reductive geometric stabilizer with tame Weyl group.    
\end{corollary}

\subsubsection{Applications to fibrations}

Suppose $\cX$ is a stack with a proper good moduli space which parametrizes a class of projective varieties, for example $\overline{\mathcal{M}}_g$ or the KSBA moduli space of stable varieties, the $K$-moduli stack parametrizing $K$-semistable Fano varieties, or the moduli space of boundary polarized Calabi-Yau surface pairs \cite{ABBDILW, blum2024goodmodulispacesboundary}. Then the root stack valuative criterion yields certain nice integral models for fibrations and a method for classifying the singular fibers of fibrations whose general fiber is parametrized by the stack $\cX$ as follows. 

Let $Y_\eta \to \spec K$ be an object over of $\cX$ and let $\cY \to \cR_n$ be the extension to a family over the $n^{th}$ root stack given by Theorem \ref{thm:main}. Then taking the coarse moduli space yields an integral model $Y \to \spec R$ whose central fiber is a quotient of an object parametrized by $\cX$ by $\bmu_n$. This turns the question of classifying singular fibers to one of classifying $\bmu_n$ actions on the objects of $\cX$. See \cite{MR3961332}, \cite[Section 7]{bejleri2024heightmodulicyclotomicstacks} and \cite{javier_hyperelliptic} for examples involving stable curves. 

\begin{corollary}\label{cor:K_moduli}
 Let $R$ be a DVR containing $\mathbb{C}$. Suppose $Y_\eta \to \spec K$ is a $K$-semistable Fano variety. Then there exists a proper model $Y \to \spec R$ with klt singularities whose central fiber is the quotient of a $K$-polystable Fano variety by $\bmu_n$. In particular, the central fiber is irreducible and of multiplicity $n$. 
\end{corollary} 

Example \ref{ex:conic_tame} can be seen as a special case of Corollary \ref{cor:K_moduli} for the $K$-semistable Fano variety $\mathbb{P}^1$. 

\begin{corollary}\label{cor:bpcy_moduli} Let $R$ be a DVR containing $\mathbb{C}$. Suppose $(Y_\eta, D_\eta) \to \spec K$ is a boundary polarized Calabi-Yau surface pair in the sense of \cite{ABBDILW, blum2024goodmodulispacesboundary}. Then there exists a proper model $(Y, D) \to \spec R$ with semi-log canonical singularities whose central fiber is the quotient of a boundary polarized Calabi-Yau by a $\bmu_n$ action. In particular, $(Y,D + (Y_0)_{red})$ is semi-log canonical and $Y_0$ has multiplicity $n$. 
\end{corollary}

Moreover, by \Cref{cor_global}, Corollaries \ref{cor:K_moduli} and \ref{cor:bpcy_moduli} have analogues for fibrations over global curves. 

\subsubsection{$p$-adic integration on stacks}\label{sec:padic}

The mixed characteristic \Cref{thm_mixed_char_adequate} is motivated by the work of Groechenig, Wyss and Ziegler on $p$-adic integration on stacks \cite{groechenig2025twistedpointsquotientstacks, gwz2}. Let $F/\bQ_p$ be a $p$-adic local field with ring of integers $\cO_F$ and let $\cX \to X$ be a morphism over $\cO_F$ which is obtained from spreading out a good moduli space as in \Cref{rem:spreading}. In particular, suppose $p$ is chosen large enough so that $\cX \to X$ satisfies \Cref{thm_mixed_char_adequate}. Then \Cref{thm_mixed_char_adequate} implies that 
$$
\pi_\infty \colon \varprojlim_n \cX\left(\sqrt[n]{\Spec \cO_F}\right) \to X(\cO_F)
$$
is surjective. 

In the setting of a global quotient stack $\cX = [U/G]$, \cite{groechenig2025twistedpointsquotientstacks, gwz2} identify the fibers of $\pi_\infty$ (away from a set of measure zero) as the rational points inside a bounded convex domain in the Bruhat-Tits building of $G_F$. Under the assumption of surjectivity of $\pi_\infty$ (\cite[Assumptions 4.3(ii) and 6.1(iii)]{groechenig2025twistedpointsquotientstacks}) they construct a canonical measure $\mu_{orb}$ on $X(\cO_F)$ and give a formula for its volume in terms of an Ehrhart series associated to the special fiber of $\cX \to \spec \cO_F$ (\cite[Theorem 1.1]{groechenig2025twistedpointsquotientstacks}). Note in \emph{loc. cit.}, surjectivity of $\pi_\infty$ is proved by hand in the examples of interest there (e.g. \cite[Lemma 6.4]{groechenig2025twistedpointsquotientstacks}). \Cref{thm_mixed_char_adequate} and the surjectivity of $\pi_\infty$ open the door for these techniques to be applied to $p$-adic integration on more general moduli stacks. 

\subsection{Conventions.} By reductive group scheme $G\to S$ we mean, as in \cite[Definition 3.1.1]{ConradReductive},  a smooth $S$-affine group scheme, such that the geometric fibers are connected reductive groups. We say that a flat, separated, finitely presented group scheme $G\to S$ over a scheme $S$ is linearly reductive if $BG\to S$ is a good moduli space. When we don't specify, the cohomology will be the fppf cohomology.
We will pass between an algebraic stack over a base scheme $B$ to the corresponding category fibered in groupoids over $\cS ch(B)$ without mentioning it specifically; we expect that this will not cause any confusion. Whenever we mention the site of an algebraic stack $\cX$, it will be the flat-fppf site: its object are flat morphisms $B\to \cX$ from a scheme, and a covering of an object is an fppf-covering. All algebraic stacks will be noetherian, unless mentioned otherwise. All algebraic stacks in the body of the paper will have affine diagonal. The $n^{th}$ root stack of a DVR $R$, denoted $\sqrt[n]{\spec R}$ or $\cR_n$, is the algebraic stack $[\spec R[t^{1/n}]/\bmu_n]$ where $t$ is a uniformizer of $R$. 

\subsection{Acknowledgements} We thank Jarod Alper, Harold Blum, Patrick Brosnan, Elden Elmanto, Tom Haines, Sam Molcho, Swarnava Mukhopadhyay, Danny Ofek, Zinovy Reichstein, Jason Starr for helpful conversations. We thank Andres Fernandez Herrero for helpful conversations and for giving us detailed feedback on a previous version of this manuscript. We especially thank Sid Mathur for many helpful conversations about gerbes. We thank Michael Groechenig for asking about mixed characteristic adequate moduli spaces which led to \Cref{thm_mixed_char_adequate} and David Rydh for a helpful conversation regarding the proof of \Cref{thm_reduction_to_gerbe_general_case}. We thank K\k{e}stutis Česnavičius for pointing out a gap in a previous argument and for suggesting a fix for it. Claude (Anthropic) was used to identify potential gaps in our work and assist us in resolving the minor issues identified. All results and ideas are due to the authors, who take full responsibility
for the content of the paper.   

\section{Groups, torsors and gerbes on algebraic stacks}
In this section we develop some preliminary results on reductive groups and their torsors and gerbes over an algebraic stack. We also refer the reader to \cite{good_moduli_gerbe} where this material will be developed in more depth. 

\subsection{Reductive groups over a stack}\label{subsection_red_groups}
We begin with the following:
\begin{definition}
    Let $\cX$ be an algebraic stack. A group $G$ over $\cX$ is a sheaf on $\cX_{\fppf}$ which is a sheaf in groups, and such that there is a fppf-cover $U\to \cX$ such that the sheaf $G|_U$ is represented by a scheme over $U$. We say that $G$ is a smooth (resp. separated, of finite type, fppf) group over $\cX$, if there is an fppf-cover as above such that $G|_U$ is represented by a smooth (resp. separated, of finite type, fppf) group scheme over $U$.
\end{definition}

\begin{definition}
    Let $G\to \cX$ be a smooth separated group over an algebraic stack $\cX$. We say that $G$ is reductive (resp. linearly reductive, a torus) if for every fppf-cover $B\to \cX$ from a scheme $B$, the pull-back $G\times_\cX B\to B$ is a reductive (resp. linearly reductive, a torus) $B$-group. 
\end{definition}
\begin{lemma}\label{lemma_moduli_of_max_tori}
    Let $G\to \cX$ a reductive group over an algebraic stack as above, with connected fibers. Then there is an algebraic stack $\cT$, with a smooth affine and representable morphism $\cT\to \cX$ such that for every morphism $f:B\to \cX$, there is an equivalence 
    \[
\operatorname{Hom}_\cX(B,\cT) \longleftrightarrow \{\text{maximal tori of }G\times_\cX B\}.
    \]
\end{lemma}
Recall that a subgroup $T \subset G\times_\cX B$ which is a torus is a maximal torus if $T_{\bar{b}} \subset G_{\bar{b}}$ is a maximal torus for all geometric points $\bar{b}$ of $B$. 
\begin{proof}
    Note that $B\mapsto \{\text{maximal tori of }G\times_\cX B\}$ is an fppf sheaf. Indeed, the restriction of this functor to $B_{\fppf}$ is represented by the scheme $\Tor_{G_B/B}\to B$, as proven in \cite[Theorem 3.2.6]{ConradReductive}. Thus, $\cT$ is a sheaf over $\cX$. Then to prove representability, it is enough to show $\cT\to\cX$ is affine. For this, we may look fppf locally, where we can assume $\cX$ is a scheme; we then apply  \cite[XII, 5.4]{SGA3}.%; note Brian's 3.2.6 only proves quasi-affine.
\end{proof}

\begin{lemma}\label{lemma_trivializing_the_torus_is_etale}
    Let $\cX$ be an algebraic stack, and 
    $T\to \cX$ a torus over $\cX$ which fppf-locally
    is isomorphic to
    $\Gm^n$. Then the sheaf
    $\underline{\Isom}_{\operatorname{grp}}(\Gm^n\times \cX,T)$ over $\cX_{\fppf}$, that sends $U\to \Isom_{\operatorname{grp}}(\Gm^n(U),T(U))$, is represented by an algebraic stack, which is \'etale over $\cX$ but not noetherian for $n>1$.
\end{lemma}
\begin{proof}
    We see that $\underline{\Isom}_{\operatorname{grp}}(\Gm^n\times \cX,T)$ is a $\GL_{n,\cX}(\bZ)=\operatorname{Aut}(\bG_{m,\cX}^n)$-torsor. Since $\GL_{n,\cX}(\bZ)$ is \'etale over $\cX$, it follows that $\underline{\Isom}_{\operatorname{grp}}(\Gm^n\times \cX,T)$ is \'etale over $\cX$. Note it is not noetherian as $\GL_{n,\cX}(\bZ)$ is not noetherian.
\end{proof}

The following lemma is a generalization of \cite[Theorem 3$''$]{serre2010bounds}; see also \cite[Proposition 5.10]{gille2024loop}. 

\begin{lemma}\label{lemma_there_is_a_bmu_point}
Let $G$ be a reductive group over a field $k$ and suppose $\bmu_r$ acts on $G$. Then there exists a $\bmu_r$-invariant maximal torus of $G$ defined over $k$. 
\end{lemma}

\begin{proof} If $G$ is a torus there is nothing to prove, so we will assume $G$ is not a torus. Suppose $k$ has characteristic exponent $p \geq 1$ and write $r = p^am$ where $m$ is coprime to $p$. Following the proof of \cite[Theorem $3''$]{serre2010bounds}, we induct on $m$ and the dimension of $G$. Whem $m = 0$, $\bmu_r = \bmu_{p^a}$ is a connected group scheme. Let $\varphi : \bmu_r \to \Aut(G)$ be the morphism inducing the action. Since $\Out(G)$ is \'etale over $k$ \cite[Theorem 7.1.9(1)]{ConradReductive}, the composition $\bmu_r \to \Out(G)$ is trivial so $\varphi$ must factor through the adjoint group $G^{ad} = G/Z(G)$. By \cite[Lemma 7.1]{martens2015}, the morphism $\varphi : \bmu_r \to G^{ad}$ factors through a torus which we may take to be maximal by Grothendieck's theorem on maximal tori (see \cite[Remark A.1.2]{ConradReductive}). Lifting this maximal torus of $G^{ad}$ to $G$, we obtain a maximal torus of $G$ which is $\bmu_r$-invariant as required. If $m > 0$, then $\bmu_r \cong \bmu_{p^a} \times \bmu_m$ where $\bmu_m$ is etale of order coprime to the characteristic.
Then the invariant subgroup of the $\bmu_m$ action $G^{\bmu_m}$ satisfies the conclusions of \cite[Proposition 3]{serre2010bounds}. Indeed as in \emph{loc. cit.}, the conclusions can be checked after passing to $\bar{k}$ in which case $\bmu_m \cong \bZ/m\bZ$ where \emph{loc. cit.} applies directly. In particular, $G_1 = (G^{\bmu_m})^0$ is reductive and it is positive dimensional if $G$ is not a torus.
Now we proceed exactly as in the end of the proof of \cite[Theorem $3''$]{serre2010bounds} with $A$ replaced by $\bmu_r$ and $A' = <s>$ by $\bmu_m$; we report the salient steps below. We consider the $\bmu_r/\bmu_m = \bmu_{p^a}$ action on $G_1$ and take a maximal $\bmu_{p^a}$-invariant torus in $G_1$, which exists by the base case $m = 0$ above, and we take $G_2$ its centralizer. Observe that $\bmu_r$ acts on $G_2$. As $\dim(G_2)<\dim(G)$, by induction $G_2$ admits a $\bmu_r$-maximal torus. This will have the same dimension as a maximal torus for $G$; thus we have found a maximal torus of $G$ (contained in $G_2$) which is fixed by $\bmu_r$. 
\end{proof} 

\begin{remark}The proof shows more generally that the conclusion of \cite[Theorem $3''$]{serre2010bounds} holds whenever $A$ is a linearly reductive finite group scheme with a composition series $1 = A_0 \subset A_1 \ldots \subset A_l = A$ such that $A_i$ is normal in $A$ and $A_{i+1}/A_i$ is a $k$-form of either $\bmu_r$ or $\bZ/r\bZ$ for some $r$. \end{remark}

\begin{proposition}\label{prop_there_is_a_max_torus}
    Let $R$ be a henselian DVR, and let $\cR$ be the $r^{th}$ root stack of $\spec R$ at the closed point. Let $G\to \cR$ be a reductive group with connected fibers. Then there is a torus $T\to \cR$ with a morphism $i:T\subseteq G$ such that, for every geometric point $\spec l \to \cR$, the fiber $T_{l}\to G_{l}$ is a maximal torus.
\end{proposition}

\begin{proof} For any map of algebraic stacks $f : \cX' \to \cX$ let $\Gamma(\cX'/\cX)$ denote the groupoid of sections of $f$. Let $\cT \to \cR$ be the stack of maximal tori of $G$ which is smooth and affine over $\cR$ by Lemma \ref{lemma_moduli_of_max_tori}. We wish to show that $\Gamma(\cT/\cR)$ is nonempty. Let $\cR_0 = B\bmu_r$ be the residual gerbe of the closed point of $\Spec R$ and let $\cT_0 = \cT \times_{\cR} \cR_0, G_0 = G \times_{\cR} \cR_0$. The data of $G_0$ is equivalent to the data of a reductive group $G_{k}$ over $k$, the residue field of $R$, together with a $\bmu_r$-action on it. Observe that $G_0$ is reductive as $G\to \cR$ is reductive, and the geometric fiber of $G\to \cR$ over the closed point agrees with the geometric fiber of $G_0\to \spec(k)$. By Lemma \ref{lemma_there_is_a_bmu_point}, there exists a $\bmu_r$-invariant maximal torus of $G_k$, i.e. the groupoid of sections $\Gamma(\cT_0/\cR_0)$ is nonempty. Since $R$ is henselian, the pair $(\cR, \cR_0)$ is henselian by \cite[Theorem 3.6]{AHRetalelocal}, and $\cR$ has the resolution property as it is a global quotient of an affine scheme by $\bG_m$. Thus by \cite[Proposition 7.9]{AHRetalelocal}, $\Gamma(\cT/\cR) \to \Gamma(\cT_0/\cR_0)$ is essentially surjective, and in particular nonempty. 
\end{proof} 

\subsection{Gerbes and torsors over stacks}

We recall some background on torsors, gerbes and bands. Let $\cX$ be an algebraic stack and $G/\cX$ be a group stack as in the previous section. 

\begin{definition}
    A $G$-\emph{torsor} over $\cX$ is an fppf sheaf $P \to \cX$ with an action $G \times_{\cX} P \to P$ for which there exists an fppf cover $U \to \cX$ such that $P_U \cong G_U$ with the left action of $G_U$. A \emph{gerbe} over $\cX$ is an fppf stack $\cY \to \cX$ for which there exists an fppf cover $U \to \cX$ and a group $G_U \to U$ such that $\cY_U \cong \cB G_U$ where $\cB G_U$ is the classifying stack of $G_U \to U$.  
\end{definition}

We refer to \cite[Chapter IV]{giraud} for background on bands which we recall here; see also \cite[Section 3.1]{edidin2001brauer}. Let $\cG \to \cX$ be a gerbe which is trivialized after an fppf cover $U \to \cX$ and suppose that the group $G_U$ is pulled back from a group stack $G \to \cX$. Over $U \times_{\cX} U$ we have descent data for $BG_U$ which is a section of the stack $\Isom_U(\cB G_U, \cB G_U) = [\Aut( G_U)/G_U]$ where $G_U$ acts by conjugation. Passing to $\pi_0$ yields a section of $\Out(G_U)$ over $U\times_{\cX} U$ satisfying the cocyle condition, i.e. an element of $\oH^1(\cX, \Out(G))$. Note that there is a natural map $\oH^1(\cX, \Aut(G)) \to \oH^1(\cX, \Out(G))$ and that elements of $\oH^1(\cX, \Aut(G))$ classify $\cX$-forms of $G$. 

\begin{definition}
        Let $\cG \to \cX$ be a gerbe 
        as above. The \emph{band} of $\cG$ is the element $\Band(\cG) \in \oH^1(\cX, \Out(G))$ described above. Similarly, if $G' \to \cX$ is a form of $G$, we define the band of $G'$ as the image $\Band(G') \in \oH^1(\cX, \Out(G))$ of the class of $G'$ under the natural map $\oH^1(\cX, \Aut(G)) \to \oH^1(\cX, \Out(G))$. We say \emph{$\cG$ is banded by $G'$} if $\Band(\cG) = \Band(G')$ and call $\cG$ a $G'$-gerbe. 
\end{definition}

\begin{comment}
\begin{theorem}
    Let $\cX$ be an algebraic stack and $G$ a sheaf of groups on $\cX_{fppf}$. Then there is a bijection between $\oH^1(\cX,G)$ and isomorphism classes of $G$-torsors on $\cX$. Similarly, if $G$ is abelian, there is a bijection between $\oH^2(\cX,G)$ and isomorphism classes of $G$-gerbes on $\cX$.
\end{theorem}
\end{comment} 

When $G$ is abelian, $G$-torsors and $G$-gerbes are classified by fppf cohomology $\oH^1(\cX,G)$ and $\oH^2(\cX,G)$ respectively (e.g. \cite[Theorems 12.1.5 \& 12.2.8]{olsson2016algebraic}). When $G$ is non-abelian,  we follow \cite{giraud} and  define $\oH^1(\cX, G)$ and  $\oH^2(\cX,G)$ respectively as the set of isomorphism classes of $G$-torsors and $G$-gerbes on $X$. The following is a generalization of \cite[Ch. V, sect. 3.2, p. 75]{Douai} to stacks. 

\begin{proposition}\label{prop_every_gerbe_is_banded} Let $\cG \to \cX$ be a gerbe over an algebraic stack $\cX$. Suppose that there exists an fppf cover $U \to \cX$ such that $\cG|_U \cong \cB G_U$ where $G_U \to U$ is reductive and $U$ is connected. Then there exists a reductive group $G \to \cX$ and a form $G' \to \cX$ of $G$ such that $\Band(G') = \Band(\cG)$.     
\end{proposition}

\begin{proof} Since $U$ is connected, the type of $G_U \to U$ is constant \cite[XXII Section 2]{SGA3III}. Let $G_0 \to \Spec \Z$ be the Chevalley group of this type, that is, the unique split reductive group with the given root data \cite[Theorem 6.1.17]{ConradReductive}. After possibly refining $U$, we may assume that $G_U \cong G_0 \times_\Z U$. Let $G := G_0 \times_{\Z} \cX$. Then $\Band(\cG) \in \oH^1(\cX, \Out(G))$. On the other hand, for the Chevalley model, $\Aut(G_0) \to \Out(G_0)$ splits \cite[Theorem 7.1.9 (3)]{ConradReductive}. Fix such a splitting. This induces a splitting $\Aut(G) \to \Out(G)$ over $\cX$ and thus a lift of $\Band(G)$ to an element of $H^1(\cX, \Aut(G))$. Let $G'$ be the corresponding $\cX$-form of $G$. This is our desired group scheme. 
\end{proof}

\section{Preliminary reductions}

In this section we make some preliminary reductions that show it suffices to prove Theorem \ref{thm:main} in the case where the stack is a gerbe whose geometric automorphism groups are reductive. Throughout this section, we let $R$ be a DVR with fraction field $K$ and residue field $k$ and we assume that $\cX$ is an algebraic stack with good moduli space $\cX \to \spec R$. 

\subsection{Passing to the polystable locus} We first show that it suffices to prove Theorem \ref{thm:main} in the case where the $K$-point of $\cX$ is polystable (see e.g. \cite[Definition 3.2 and Lemma 3.3]{bejleri2024proper}).  We begin with the following.

\begin{lemma}\label{lemma_kempf_for_lin_reductive}
    Let $V = \spec A \to \Spec K$ with a $\GL_n$ action such that $[V/\GL_n] \to \spec K$ is a good moduli space. Let $x \in V(K)$ be any $K$-point. Then there is a one parameter subgroup $\lambda:\Gm\to \GL_n$ defined over $K$ such that $\lim_{t\to 0}\lambda(t)x$ exists and belongs to the unique closed orbit of $V$.
\end{lemma}
\begin{proof}
For perfect fields this follows from \cite{kempf1978instability}, especially \cite[\S 4]{kempf1978instability}.
In characteristic $p$ instead we proceed as follows. By  \cite[Theorem 1.3]{bate2017cocharacter}, the cocharacter closure of $\GL_n(K)x$ contains a unique cocharacter closed $\GL_n(K)$ orbit and there is a $\lambda: \mathbb{G}_m \to \GL_n$ defined over $K$ such that $\lim_{t \to 0}\lambda(t)x$ exists and belongs to this cocharacter closed orbit. We wish to show that this limit in fact belongs to the closed orbit.
Thus it suffices to check that every cocharacter closed orbit of the $\GL_n(K)$ action on $V$ is contained in the unique closed orbit of $V$. We refer the reader to \cite{bate2017cocharacter} for the definition of cocharacter closed and cocharacter closure. Toward this end, let $z\in V$ be a point in a cocharacter-closed orbit. After passing to the separable closure $K^s$ of $K$, we can write $[V/\GL_n] \times_K K^{s} \cong [W/G]$ where $W$ is affine and $G \subset \GL_n$ is linearly reductive; this can be arranged using \cite[Theorem 6.1]{AHRetalelocal}. By assumption, $G$ is linearly reductive so its connected component is a torus $T \cong \mathbb{G}_{m,K^s}^n$. This follows from \cite{nagata1961complete} coupled with the following two facts. First, by definition, being a torus can be checked \'etale locally, see \cite[Definition 3.1.1]{ConradReductive} and \cite[Background]{ConradReductive} for the definition of multiplicative type. Second, the connected component of the identity of a group over $K^s$ base-changes to the algebraic closure of $K$ to the connected component of the identity; see \cite[Exercise 3.2.11]{QingLiu2002algebraic}. By \cite[Theorem 1.5(ii)]{bate2017cocharacter}, the orbit of $z$ remains cocharacter closed over $K^s$ and under the isomorphism with $[W/G]$ corresponds to a cocharacter closed orbit of $G(K)$. Again by \cite{kempf1978instability}, after passing further to the algebraic closure $\overline{K}/K^{sep}$, there is a 1-parameter
subgroup $\lambda_{\overline{K}}:(\Gm)_{\overline{K}}\to (\Gm^n)_{\overline{K}}$ such that $\lim_{t\to 0}\lambda_{\overline{K}}(t)z$ exists and is contained in the closed orbit. But $\lambda$ is of the form $t\mapsto (t_1^{a_1},\ldots,t_n^{a_n})$ with $a_i\in\bZ$, so it is the pullback of a character $\lambda_{K^s}:(\Gm)_{K^s}\to (\Gm^n)_{K^s}$. Thus $\lim_{t \to 0}\lambda_{K^s}(t)z$ exists and is contained in the closed orbit. Since the orbit of $z$ is cocharacter closed, this limit must also be contained in $G(K)z$ so $z$ is contained in the closed orbit of $V\times_K K^s$. As belonging to the closed orbit can be checked after a separable field extensions, then $z$ is also in the closed orbit of $V$, as desired. 
\end{proof}

\begin{proposition}\label{prop_knowing_the_result_for_poli_pts_suffices}
Let $R$ be a DVR with generic point $\eta = \spec K$. Consider an algebraic stack $\cX$ with a good moduli space $\xi:\cX\to \spec R$. Assume that
\begin{enumerate}
\item $\cX$ admits a $K$-point $\phi:\eta\to \cX$, and 
\item every $\psi:\eta\to \cX$ sending $\eta$ to the polystable point of $\cX_\eta$ extends to a section over some root stack $\cR = \sqrt[r]{\spec R}$ of $\spec R$.
\end{enumerate}
Then $\phi$ extends to a section over some root stack $\cR$ of $\spec R$. 
\end{proposition}
\begin{proof}
We first prove that  there is a map $\Theta_K \colon = [\bA^1_{K}/\Gm]\to \cX$ sending $1$ to $\phi(\eta)$ and $0$ to the unique polystable point of $\cX_K$ over $\eta$. Indeed, from \cite[Theorem 6.1]{AHRetalelocal}, we can write $\cX_K = [V/\GL_n]$ with $V$ affine and $\cX_K \to \eta$ a good moduli space. By \Cref{lemma_kempf_for_lin_reductive}, there is a 1-parameter subgroup defined over $K$ such that $\lambda : \mathbb{G}_m \to \GL_n$ such that $\lim_{t \to 0}\lambda(t)\phi(\eta)$ is in the unique closed orbit of the $\GL_n$ action on $V$. Taking quotients of the induced map $\mathbb{A}^1_K \to \cX_K$ yields the required map $\Theta_K \to \cX$.

By assumption, up to replacing $\spec R$ with a root stack $\cR$, we can extend $\spec K\xrightarrow{\iota}\bA^1_{K}\to\cX$ to a morphism $\cR\to \cX$, where $\iota$ is the inclusion of the origin in $\bA^1_{K}$. Consider then the following diagram
\[
\xymatrix{\spec K\ar[r] \ar[d]_\iota & \spec R\ar[d]^\beta \\ \bA^1_{K}\ar[r]^\alpha & \spec A.}
\]
where $A$ is the fiber product of $K[t]\times_{K} R$. More explicitly, 
\[
A = \{p(t)\in K[t]:p(0)\in R\} = R\left[t, \frac{t}{\pi}, \frac{t}{\pi^2},\ldots\right] \]
where $\pi$ is a uniformizer in $R$. As $\spec K\to \bA^1_{K}$ is a closed embedding and $\spec K\to\spec R$ is affine, $\spec A$ is a pushout of the previous diagram in algebraic stacks by \cite[Theorem 4.2]{alper2024artin}. Consider the $n$-th root stack of $\spec A$ at the Cartier divisor $\pi=0$. This leads to the following cartesian diagram:

\[
\xymatrix{\spec K\ar[r] \ar[d]^\iota & \cR\ar[d]^b \\ \bA^1_{K}\ar[r]^a & \sqrt[n]{\spec A}.}
\]
We argue now that this diagram is also a pushout in algebraic stacks. By \cite[Theorem 4.2]{alper2024artin} it suffices to check that $\cO_{\sqrt[n]{\spec A}}\to a_*\cO_{\bA^1_{K}}\times_{(a\circ\iota)_*\cO_{\spec K}} b_*\cO_\cR$ is an isomorphism. This can be checked flat-locally on $\sqrt[n]{\spec A}$, so consider the cover $\spec A[r]/(r^n-\pi)\to \sqrt[n]{\spec A}$. This cover pulls back to the diagram of schemes on the left, leading to the diagram of rings on the right
\[
\xymatrix{\spec K[r]/(r^n-\pi) \ar[r] \ar[d] & \spec R[r]/(r^n-\pi) \ar[d] \\ \spec K[t,r]/(r^n-\pi)\ar[r] & \spec A[r]/(r^n-\pi)}\text{ } \text{ }\text{ }\text{ } \text{ }\xymatrix{K[r]/(r^n-\pi) & R[r]/(r^n-\pi)\ar[l] \\ K[t,r]/(r^n-\pi)\ar[u]^-{t=0} & A[r]/(r^n-\pi)\ar[u]\ar[l].}
\]
It suffices to note that the right-most diagram is a pull-back in rings. 

By the universal property of pushouts, we have a map $\sqrt[n]{\spec A} \to \cX$. Since 
\[
A=\textrm{colim}_l R\left[\frac{t}{\pi^m}\mid m\leq l\right]=\textrm{colim}_l R\left[\frac{t}{\pi^l}\right]
\]
and $\cX$ is of finite type, by the functorial characterization of locally finitely presented morphisms, the map $\sqrt[n]{\spec A}\to \cX$ factors via $\sqrt[n]{\spec R[\frac{t}{\pi^k}]} = \sqrt[n]{\spec R[s,t]/(s\pi^k-t)}\to \cX$ for some $k\gg 0$. Consider the morphism $i:\spec R\to \spec R[s,t]/(s\pi^k-t)$ induced by the $R$-algebra homomorphism $j:R[s,t]/(s\pi^k-t)\to R$ that sends $s\mapsto \pi$ and $t\mapsto \pi^{k+1}$. We have that:
\begin{enumerate}
    \item the image $i(\eta)$ of the generic point of $\spec R$ is contained in $D(t)\cap D(\pi) = \mathbb{A}^1_K\setminus \{0\}$ as $j(t\pi)\neq 0$
    \item the image $i(x)$ of the closed point of $\spec R$ is the prime ideal $(\pi, s, t)$, which agrees with the image of the closed point of $\spec R$ via $\beta$.
\end{enumerate}
Then it induces a morphism $\sqrt[n]{\spec R}\to \sqrt[n]{\spec R[s,t]/(s\pi^k-t)}$, which if postcomposed with the morphism $\zeta\colon \sqrt[n]{\spec R[s,t]/(s\pi^k-t)}\to \cX$ yields a morphism $\sqrt[n]{\spec R} \to \cX$. Observe that the restriction to the generic point of $\zeta$ is isomorphic to $\phi \colon \eta \to \cX$ by $(1)$, as any two points in $\bA^1_K\smallsetminus \{0\}$ map to isomorphic points in $\cX$. Similarly, the morphism $\zeta$ sends the special point to a point over the closed point of $\spec R$ by $(2)$.
\end{proof}

\subsection{Reduction to the case of a gerbe}

The next result shows that it suffices to prove Theorem \ref{thm:main} for gerbes over root stacks.
In the following theorem we will need the definitions of (poly)stable point and stable good moduli space; we refer the reader to \cite[Definition 2.5]{ER} and \cite[Definition 3.2]{bejleri2024proper}.

\begin{theorem}\label{thm_reduction_to_gerbe_general_case}Let $\cX$ be a reduced algebraic stack with a separated good moduli space $\pi\colon \cX\to X$. Then there exists a diagram
\[
\xymatrix{\cZ \ar[r] \ar[d] & \cX \ar[d] \\ Z \ar[r] & X}
\]
with vertical maps good moduli spaces satisfying 
\begin{enumerate}[(1)]
\item $Z \to X$ is a proper surjection,
\item for any field $K$, every polystable $K$-point of $\cX$ lifts to a $K$-point of $\cZ$, and
\item $\cZ \to Z$ factors as $\cZ\xrightarrow{\rho} \cY\to Z$ where $\cY \to Z$ is a tame coarse moduli space and $\rho$ is a disjoint union of gerbes with smooth, connected and linearly reductive geometric stabilizers. 
\end{enumerate}
\end{theorem}

\begin{proof} We first construct an auxiliary algebraic stack $\cX'$ by induction on the dimension of $X$. Each irreducible component of $\cX'$ will be a closed substack of $\cX$ such that $\cX'\to\cX\to X$ is surjective and every $x\in X$ will have a preimage in $\cX'$ which is a stable point in the sense of \cite[Definition 2.5]{ER}. If $X$ has dimension 0, we take $\cX'$ to be unique closed substack of $\cX$ which is a gerbe over $X$, i.e. the residual gerbe at the closed point of $\cX$. If $X$ has dimension $n$, we first take the generic points of $X$, which we denote by $\eta_1,...,\eta_m$. There is a unique substack of $\cX\times_X\eta_i$ which is closed in $\cX\times_X\eta_i$ and which is a gerbe over $\eta_i$, we take its scheme-theoretic closure $\cX_i\subseteq \cX$ in $\cX$. This is a closed substack of $\cX$ which has a good moduli space by \cite[Lemma 4.14]{Alper}, denoted $\cX_i \to X_i$. moreover $\cX_i \to X_i$ is a stable good moduli space by construction. Note that $X_i$ are the irreducible components of $X$. Consider 
$$
W_i = \left\{p\in X_i \mid \text{the fiber of $\cX_i\to X_i$ contains more than one point}\right\}.
$$
This locus is closed in $X_i$ by \cite[Proposition 2.6]{ER}, and of dimension strictly less than $\dim X_i$ by construction of $\cX_i$. Consider the closed substack $\cW_i:=(W_i\times_X\cX)_{red} \subset \cX$. By induction, there exists an algebraic stack $\cW_i'\to \cW_i$ satisfying the requirements and we can take $\cX':= \cX_1 \sqcup \ldots \cX_m \sqcup \cW_1' \sqcup \ldots \cW_m'$.
We have that:
\begin{enumerate}
    \item $\cX'$ admits a good moduli space by \cite[Lemma 4.14]{Alper} as each of its components are closed substacks of $\cX$; let $\cX'\to X'$ be this good moduli space.
    \item The map $X'\to X$ is finite. This follows since by construction $\cX'$ is a finite union of closed substacks of $\cX$; and since good moduli spaces are universally closed, see \cite[Theorem 4.16 and its proof]{Alper}. Indeed, observe that in \textit{loc. cit.} it is proven that for a closed substack $\cC\subseteq \cX$, the scheme-theoretic image of $\cC$ in $X$, which is a closed subscheme of $X$ by definition, is the good moduli space for $\cC$.
    \item Every polystable $K$-point of $\cX$ admits a lift to a stable $K$-point of $\cX'$.
    \item The map $\cX'\to X'$ is a stable good moduli space, as every connected component of $\cX'$ a dense subset of stable points.
\end{enumerate}
 Thus we may apply the canonical reduction of stabilizers \cite{ER} which yields a modification $\cZ \to \cX'$ such that:
\begin{enumerate}[(a)]
    \item $\cZ\to \cX'$ is an isomorphism over the set of points which are stable,
    \item the dimension of the stabilizers of the geometric points of $\cZ$ is locally constant, and
    \item $\cZ$ admits a good moduli space $\cZ\to Z$ which is obtained by a sequence of blow-ups of $X'$. In particular, the composition $Z\to X'\to X$ is proper.
\end{enumerate}
By \cite[Proposition B.2]{ER} and its proof (especially \cite[Proposition B.6]{ER}, see also \cite[Proposition 2.3]{bejleri2024proper}) we can factor $\cZ\to Z$ as $\cZ\to \cY \to Z$ where the first map is a disjoint union of gerbes whose geometric automorphism groups are the smooth connected components of the identity of the geometric automorphism
groups of $\cZ$, and $\cY \to Z$ is a tame coarse moduli space. By \cite[Proposition 3.13]{Alper} the map $\cZ \to \cY$ is also a good moduli space \'etale locally, so the geometric automorphism groups are linearly reductive. 
\end{proof} 

In the setting of \Cref{thm_reduction_to_gerbe_general_case}, let $R$ be a DVR and $\eta = \spec K$ be its generic point. Consider a diagram
$$
\xymatrix{
\eta \ar[r]^\phi \ar[d] & \cX \ar[d] \\
\spec R \ar[r] & X
}
$$
where $\phi$ is a polystable $K$-point.

\begin{corollary}\label{cor_intermediate}
There exists a root stack $\cR \to \spec R$ and a commutative diagram 
$$
\xymatrix{\eta \ar[r]^{\phi_Z} \ar[d] & \cZ \ar[d] \\ 
\cR \ar[r] \ar[d] & \cY \ar[d]\\
\spec R \ar[r] & X}
$$
where the composition $\eta \to \cZ \to \cX$ is $\phi$. 
\end{corollary}
\begin{proof}

The map $\phi$ lifts to a morphism $\phi_\cZ\colon \eta\to \cZ$ by \Cref{thm_reduction_to_gerbe_general_case}(2), and $\spec R\to X$ lifts to $\spec R \to Z$ by \Cref{thm_reduction_to_gerbe_general_case}(1) and the valuative criterion for properness. 

Now $\cY \to X$ is a proper tame map by \Cref{thm_reduction_to_gerbe_general_case}(3) and so by the root stack valuative criterion for proper tame maps \cite[Theorem 3.1]{bresciani2024arithmetic} the composition $\eta \to \cZ \to \cY$ extends to a map $\cR \to \cY$ from some root stack of $\spec R$ commuting with $\spec R \to X$ as required. 
\end{proof} 

For most of the paper, we will only need the following corollary. 

\begin{corollary}\label{cor_reduction_to_gerbe}
    Let $R$ be a DVR, with generic point $\eta = \spec K$. Consider a stack $\cX$ with good moduli space $\pi:\cX \to \spec R$, and suppose $\phi:\eta \to \cX^{ps} \subset \cX$ is a polystable $K$-point. Then there exists a root stack $\cR = \sqrt[n]{\spec R}$ and a diagram as follows
    \[
    \xymatrix{&\cG\ar[r]^p \ar[d]^-{\pi'} & \cX\ar[d]^-\pi \\ \eta\ar[r]\ar[ru]^{\phi'} & \cR\ar[r] & \spec R}
    \]
    where $\phi \cong p\circ \phi'$ and $\pi'$ is a gerbe banded by a smooth, connected, linearly reductive group scheme $G \to \cR$. Moreover, the closed point of $\cG$ maps to a polystable point in $\cX$.
\end{corollary}
\begin{proof}
    The desired statement remains unchanged if we replace $\cX$ with its reduced structure. Then we can apply \Cref{cor_intermediate} to obtain a root stack $\cR$, and lifts $\cR \to \cY$ and $\phi_{\cZ} : \eta \to \cZ$. Let $\cG = \cZ\times_{\cY}\cR\to \cR$. By \Cref{thm_reduction_to_gerbe_general_case}, this is a good moduli gerbe with smooth, connected, linearly reductive geometric automorphism groups. It follows from Proposition \ref{prop_every_gerbe_is_banded} that any such gerbe is banded by a smooth, connected, linearly reductive group scheme over $\cR$ (i.e. a form of $G$ over $\cR$).

    As for the claim about the closed point of $\cG$ mapping to a polystable point of $\cX$, observe that at each saturated blow-up in the algorithm of Edidin and Rydh \cite{ER}, we are taking an open substack (namely, using their notation, a \textit{saturated blow-up}) of the blow-up of the polystable point over the closed point of $\cR$. Hence it suffices to show that the saturated blow-up intersects the exceptional divisor of the blow-up, and each blow-up in the Edidin-Rydh algorithm is centered at a point on the exceptional divisor of the previous blow-up. Both these statement follow from \cite[Proposition 3.7 (3)]{ER}: the first one follows from \cite[Proposition 3.7 (3)]{ER} directly, the second one follows as at each step we are blowing up the polystable point, which is the unique closed point (over the closed point of $\spec R$). As the exceptional divisor is closed, it will contain the unique closed polystable point. Then any point on the last exceptional divisor we extract will map to the polystable point of $\cX$.

\end{proof}

The situation \Cref{thm_reduction_to_gerbe_general_case} and its corollaries is summarized in the diagram below. 

\[
\xymatrix{
 & \cG \ar[r] \ar[d] \ar@{}[dr]|(.2){\big\lrcorner}  & \cZ\ar[dr] \ar[d]_{\operatorname{gerbe}} & \\
 & \cR\ar[d] \ar[r] & \cY \ar[dr]_{\operatorname{tame}} & \cX\ar[d]\\
 \spec K\ar[uur]\ar[r] & \spec R  \ar[rr] & & X.
}
\]

\section{Gerbe case}
The goal of this section is to prove the root stack valuative criterion for gerbes banded by a reductive group. When the residue characteristic is positive, we need a tameness assumption on the reductive group. 

\begin{definition}\label{def_tame} Let $G \to S$ be a reductive group over a stack  $S$. We say $G$ \textit{has tame Weyl group} if for all geometric points $x : \spec k \to S$, the order of the Weyl group of $G_x$ is coprime to the characteristic of $k$.
\end{definition}

\begin{theorem}\label{thm_case_of_gerbe_over_root_stack}
    Let $R$ be a henselian DVR with fraction field $K$ and let $\cR$ be a root stack of $\spec R$. Let $\cG\to \cR$ be a gerbe banded by a reductive group $G \to \cR$ which has tame Weyl group. Suppose that $\cG$ admits a $K$-point $\phi: \spec K \to \cG$. Then $\phi$ extends to a section $\cR \to \cG$ up to replacing $\cR$ with a further root stack.
\end{theorem}

\begin{comment}
\begin{remark}
    Note that the tameness assumption is not necessarily optimal. \dori{check if tamenss + special is optimal?} For example, if $\cG = BG$ and $G$ is special (e.g. $G = \GL_n$), then any $\phi$ corresponds to the trivial torsor and so it automatically extends. On the other hand, some assumption on $G$ is neccessary. See the following example.
\end{remark}
\end{comment}

\noindent The following example illustrates that some assumption on $G$ is necessary. 

\begin{example}\label{ex:tame_is_necessary}
Let $R$ be the henselization of
$\mathbb{Z}[s,t]_{(2)}$ \textcolor{black}{with uniformizer $\pi=2$ and
residue field $k = \mathbb{F}_2(s,t)$.} Consider the conic
\textcolor{black}{$C=\{x^2+2 s y^2=tz^2\}\subset\mathbb{P}^2_K$}. This can be interpreted as a map
$\Spec K\to\mathcal{B}\mathrm{PGL}_2$. We claim that this map does not extend to any root stack of
$\Spec R$. Suppose toward contradiction that it did extend over a root stack $\mathcal{R}_n$. Then the
pullback $C'$ of $C$ to \textcolor{black}{$R'=R[\pi^{1/n}]$} has good reduction, where $R'$ has the same
residue field $k$. Let $\mathcal{C}'\to\Spec R'$ be a smooth proper model. Then by smoothness
$\mathcal{C}'$ admits a $k''$-point where $k''/k$ is a \textcolor{black}{finite} separable extension,
see [Sta21, Tag 056U]. By the henselian assumption, this extends to an $R''$-point where $R''/R'$
\textcolor{black}{is} the unique finite \'etale extension corresponding to $k''/k$. In particular,
$C''=C\times_R R''$ has a $K''$-rational point\textcolor{black}{, where $K''$ is the fraction field of
$R''$. We show that this is impossible. Note that $R''$ is a DVR with uniformizer $\varpi:=\pi^{1/n}$ and
residue field $k''$, and $1,\bar s,\bar t$ are still linearly independent over $(k'')^2=\{\alpha\in k''$ which are squares of elements in $k''\}$.}

\textcolor{black}{Write the $K''$-point as $[x:y:z]$ with $x,y,z\in R''$ not all divisible by $\varpi$, so
that
\begin{equation}\label{eq:1}
x^2+\varpi^n s\,y^2=t\,z^2 .
\end{equation}
Reducing \eqref{eq:1} modulo $\varpi$ gives $\bar x^2=\bar t\,\bar z^2$, hence $\bar x=\bar z=0$. So
$\varpi$ divides $x$ and $z$, and therefore $y$ is a unit. Let $f\ge 1$ be the largest integer such that
$\varpi^f$ divides both $x$ and $z$, and write $x=\varpi^f x_1$ and $z=\varpi^f z_1$. Then $x_1$ or $z_1$
is a unit, and \eqref{eq:1} becomes
\begin{equation}\label{eq:2}
\varpi^{2f}\,(x_1^2-t\,z_1^2)=-\varpi^{n}s\,y^2 .
\end{equation}
Divide \eqref{eq:2} by $\varpi^{\min(2f,n)}$ and reduce modulo $\varpi$:
\begin{itemize}
\item if $2f<n$, we get $\bar x_1^2=\bar t\,\bar z_1^2$ with $(\bar x_1,\bar z_1)\neq(0,0)$;
\item if $2f>n$, we get $\bar s\,\bar y^2=0$ with $\bar y\neq 0$;
\item if $2f=n$, we get $\bar x_1^2+\bar s\,\bar y^2=\bar t\,\bar z_1^2$ with $\bar y\neq 0$.
\end{itemize}
Each of these is a nontrivial linear relation between $1,\bar s,\bar t$ with coefficients in $(k'')^2$,
which is a contradiction. Hence $C$ has no $K''$-point, and so $C$ cannot have good reduction over any
root stack.}
\end{example}

The proof will proceed as follows. First we will prove that, up to replacing $\cR $ with a further root stack,
$\cG\cong \cB G'$ for an inner form $G'$ of $G$. To choose $G'$, we extend over a root stack the inner form of $G|_\eta$ given by the automorphism group of the $K$-point, using the adjoint group of $G$. This will be achieved by replacing $\cG$ with a gerbe over $\cR$ banded by the \textit{center} of $G$, which is an extension of a torus by a tame finite group. Using \cite{bresciani2024arithmetic} we will further reduce the problem to the case in which $G$ is a torus, which we handle explicitly. Once we know that
$\cG\cong \cB G'$, the problem will become to lift a class $c\in \oH^1(\eta,G'|_\eta)$. We will reduce again to the case of a torus by showing that $G'$ admits a maximal torus $T$, lifting $c$ to a class in $\oH^1(\eta,N|_\eta)$ where $N$ is the normalizer of $T$ in $G$, and reducing again to the case of a torus using that $N$ is an extension of a torus by a finite tame group (so using \cite{bresciani2024arithmetic}). 
\subsection{Torus case}
In this section we include a few preparatory results that will be needed later. Whenever we take a cohomology group on an algebraic stack, we will use the flat-fppf cohomology.
\begin{lemma}\label{lemma_torsion_torus_cohom}Let $\cX$ be a tame regular algebraic stack with a dense open substack that is a scheme. Then $\oH^2(\cX,\Gm)$ is torsion and if $\cX$ admits a coarse moduli space $\pi: \cX\to X$, the cokernel of $\pi^*: \oH^1(X, \bG_m) \to \oH^1(\cX, \bG_m)$ is torsion.
\end{lemma}
\begin{proof}
 We first prove that $\oH^2(\cX,\Gm)$ is torsion. It follows from the same argument as in \cite[Proposition 3.1.3.3]{lieblich2008twisted} that the morphism $\oH^2(\cX,\Gm)\to \oH^2(U,\Gm)$ is injective, where $U$ is a dense open substack which is a scheme (see also \cite[Prop. 2.23]{achenjang}). More specifically, consider a class $c\in \oH^2(\cX,\Gm)$ which restrict to the trivial class generically. This corresponds to a gerbe $\cG\to \cX$ which generically is trivial. So $\cG|_U$ has a 1-twisted line bundle $\cL_\eta$, and its push-forward via $\cG|_U\to \cG$ will be a 1-twisted line bundle. As $\oH^2(U,\Gm)$ is torsion from \cite[Corollary 3.1.3.4]{lieblich2008twisted}, also $\oH^2(\cX,\Gm)$ is torsion. To check that $\oH^1(\cX,\Gm)$ is torsion, it suffices to use \cite[Lemma 2.3.7]{abramovich2011stable} to argue that for every line bundle $\cL$ on $\cX$, a power $\cL^{\otimes N}$ descends to $X$. 
\end{proof}
\begin{proposition}\label{prop_there_is_fet_that_splits_the_torus}
    Let $R$ be a henselian DVR, let $\cR$ be a root stack of $\spec R$ and let $T\to \cR$ a smooth group stack with geometric fibers isomorphic to $\Gm^n$ for a fixed $n$. Then there is a finite \'etale morphism $\pi:\cR'\to \cR$ where $\cR'$ is also a root stack of a henselian DVR such that $\pi^{-1}T\cong \bG_m^n\times \cR'$.
\end{proposition}
In other terms, we can split the torus after a finite \'etale morphism of the root stack $\cR$. This is well-known if $\cR$ is a scheme.
\begin{proof} Let $p$ be the residue characteristic of the closed point of $\spec R$, let $\bmu_d$ the stabilizer of the closed point of $\cR$ and $\pi$ the uniformizer of $R$. If $p>0$ let $e$ and $d'$ be such that $d=p^ed'$ and $p$ does not divide $d'$, so $\bmu_d\cong \bmu_{p^e}\times \bmu_{d'}$; if $p=0$ let $d=d'$. Observe now that $\spec R[t]/(t^d-\pi)\cong \spec R[t_1,t_2]/(t_1^{d'}-\pi, t_2^{p^e}-t_1)$ and $\cR\cong [\spec R[t_1,t_2]/(t_1^{d'}-\pi, t_2^{p^e}-t_1)/\bmu_{d'}\times \bmu_{p^e}]$. We briefly expand on how this action is defined; to simplify the notations we will denote by $A:= R[t_1,t_2]/(t_1^{d'}-\pi, t_2^{p^e}-t_1)$.

There are inclusions $\bmu_{p^e}\to \bmu_d$ and $\bmu_{d'}\to \bmu_d$, leading to a morphism $\bmu_{p^e}\times \bmu_{d'}\to\bmu_d$ which is an isomorphism. Let $a$ be such that $a(d'+p^e)\equiv 1 (\operatorname{mod} d)$; in particular $a$ is a unit modulo $p$. We define the action of $\bmu_{p^e}\times \bmu_{d'}$ on $\spec A$ using the functor of points. Given a scheme $S$, the morphisms $S\to \bmu_n=\spec k[\zeta]/\zeta^n-1$ are in bijection with $\{a\in \cO_S(S):a^n = 1\}$. Similarly, morphisms $S\to \spec A$ are in bijection with ring homomorphisms $\phi\colon A\to \cO_S(S)$. Then the desired action sends
\[{((\zeta_{p^e},\zeta_{d'})\in\cO_S(S)\times \cO_S(S)\text{ such that }\zeta^{p^e}_{p^e}=\zeta^{d'}_{d'}=1, \phi\colon A\to \cO_S(S))\longmapsto \phi_{\zeta_{p^e},\zeta_{d'}}\colon A\to \cO_S(S)}\]
where:
\begin{enumerate}
    \item $\phi_{\zeta_{p^e},\zeta_{d'}}$ and $\phi$ agree when restricted to $R$,
    \item $\phi_{\zeta_{p^e},\zeta_{d'}}(t_2)=\phi(t_2)\cdot \zeta_{p^e}^a\zeta_{d'}^a$,and 
    \item $\phi_{\zeta_{p^e},\zeta_{d'}}(t_1)=\phi(t_1)\cdot \zeta_{d'}^{ap^e}$. 
\end{enumerate}
It is immediate to check that, the ring isomorphism $A\to R[t]/t^d-\pi$ which sends $t_1\mapsto t^{p^e}$ and $t_2\mapsto t$ is equivariant under the group isomorphism $\bmu_{p^e}\times \bmu_{d'}\to \bmu_d$ which sends $(\alpha,\beta)\mapsto \alpha^a\beta^a$. Then $\cR\cong [\spec A/\bmu_{p^e}\times \bmu_{d'}]$. Observe finally that, if we identify $\bmu_{p^e}$ as a subgroup of $\bmu_{p^e}\times \bmu_{d'}$ by sending $\zeta_{p^e}\mapsto \zeta_{p^e}^{\frac{1}{a}}$, then $[\spec A/\bmu_{p^e}]\cong [\spec R[t_1,t_2]/(t_1^{d'}-\pi, t_{2}^{p^e}-t_1)/\bmu_{p^e}]\cong \sqrt[p^e]{\spec R[\sqrt[d']{\pi}]}$.

So we have a space $X$ (namely, $\spec R[t_1,t_2]/(t_1^{p^e}-\pi, t_2^{d'}-t_1)$) with an action of a group of the form $G=G_1\times G_2$ (namely, $G_1 = \bmu_{p^e}$ and $G_2=\bmu_{d'}$). We can factor $X\to [X/G_1\times G_2]$ as $X\to [X/G_1]\to [X/G_1\times G_2]$. So we can then factor the fppf $\bmu_d$-cover $\spec R[t]/(t^d-\pi)\to \cR$ as \[\spec R[t]/(t^d-\pi)\to \sqrt[p^e]{\spec R[\sqrt[d']{\pi}]}\xrightarrow{\xi} \cR.\]
Observe that $\xi$ is finite and \'etale as it is a $\bmu_{d'}$-torsor, so it suffices to prove the desired statement for $R[\sqrt[d']{\pi}]$, which is still a henselian DVR, and for $\cR = \sqrt[p^e]{\spec R[\sqrt[d']{\pi}]}$: we can assume $d=p^e$ and $p > 0$. 

Observe now that if $k$ is a separably closed field of characteristic $p$, and $T$ is a torus on $\cB \bmu_{p^e}$ over $\spec k$, then $T\cong \cB \bmu_{p^e}\times \Gm^n$ for some $n$. Indeed, as $T$ splits after the fppf-cover $\spec k\to \cB \bmu_{p^e}$, $T$ is given by a homomorphism $\bmu_{p^e}\to\textrm{Aut}(T)=\GL_n(\bZ)$. 
%the class of $T$ is determined by a class in $\oH^1(\cB\bmu_{p^e}, \GL_{n}(\bZ))$, i.e. a homomorphism $\bmu_{p^e}\to \GL_{n}(\bZ)$ up to conjugation. 
As $\GL_n(\bZ)$ is discrete over $\spec k$, the only homomorphism is the trivial one.

Consider then $R^{sh}$ the strict henselization of $R$, consider the morphisms $\spec R^{sh}\to \spec R$ and $\cR^{sh}:=\cR\times_{\spec R}\spec R^{sh} $. Over $\cR^{sh}$ one can consider the isom-sheaf $\cI^{sh}:=\underline{\Isom}_{\operatorname{grp}}(\Gm^n\times \cR^{sh},T^{sh})$ of \Cref{lemma_trivializing_the_torus_is_etale} where $T^{sh}=T|_{\cR^{sh}}$. Then there is a section $(\cB \bmu_{p^e})_{k^s}\to \cI$, where $k^s$ is the separable closure of the residue field of $R$. As $\cI^{sh}\to \cR^{sh}$ is \'etale, this section lifts uniquely to any infinitesimal thickening of $(\cB \bmu_{p^e})_{k^s}$ in $\cR^{sh}$, so from \cite[Theorem 3.4]{AHRetalelocal} it lifts to a section over $\cR^{sh}$. As $\cI^{sh}$ is the pull-back of $\underline{\Isom}_{\operatorname{grp}}(\Gm^n\times \cR,T)$ which is locally of finite type over $\cR$, and as $R^{sh}$ is the limit of finite and \'etale morphisms over $R$, there is a finite and \'etale morphism $\spec R' \to \spec R$ such that $T$ trivializes when pulled back to $\cR\times_{\spec R}\spec R'$; note that this fiber product is a root stack over $\spec R'$.
\end{proof}
\begin{lemma}\label{lemma_p_fet_induces_the_trace_map_on_fppf}
    Let $p:\cY\to \cX$ be a finite and \'etale morphism of degree $d$ of algebraic stacks and $T\to \cX$ a smooth commutative scheme. Then $p$ induces a morphism from the flat-fppf topos of $\cY$ to the one of $\cX$, and there are morphisms $T\to p_*p^{-1}T$ and $p_*p^{-1}T\to T$ whose composition is the multiplication by $d$.
\end{lemma}
Observe that this is again well-known for schemes, on the small \'etale site \cite[\href{https://stacks.math.columbia.edu/tag/03SH}{Tag 03SH}]{stacks-project}. We will need the corresponding statement for the fppf-site, in the case when $T$ is smooth; for doing so the key input is \cite[Theorem 11.7]{MR244270}. 
\begin{proof}
     As $p$ is finite and \'etale, from \cite[\href{https://stacks.math.columbia.edu/tag/0GR1}{Tag 0GR1}]{stacks-project} it induces a morphism of topoi. In particular, there are functors $p_*:\cS ch(\cY_{\fppf})\to \cS ch(\cX_{\fppf})$ and $p^{-1}:\cS ch(\cX_{\fppf})\to \cS ch(\cY_{\fppf})$.
 The sheaf represented by $T\times_{\cX}\cY$ agrees with $p^{-1}T$. Moreover, there is an adjunction morphism $\operatorname{Id} \to p_*p^{-1}$ that gives a map $T\to p_*p^{-1}T$. We claim that, as $T$ is smooth, there is also a map $p_*p^{-1}T\to T$ such that the composition $T\to p_*p^{-1}T\to T$ is the multiplication by $d$. 

 To check this, it suffices to prove that for every flat morphism $U\to \cX$ from a separated scheme
 \begin{enumerate}
     \item there is a map $p_*p^{-1}T(U)\to T(U)$, such that if precomposed with $T(U)\to p_*p^{-1}T(U)$ gives the multiplication by $d$, and
     \item this map is functorial in $U$ (i.e. it is a map of sheaves).
 \end{enumerate}
 For (1), consider then the restriction morphisms $r_U\colon \cS ch(U_{\fppf})\to \cS ch(U_{\text{\'et}})$ and $r_V\colon \cS ch(V_{\fppf})\to \cS ch(V_{\text{\'et}})$ where $V=U\times_\cX\cY$,
 and the subscript \'et stands for the small
 \'etale site. Then:
 \begin{itemize}
     \item $T(U)=r_U(T|_U)(U)$. Indeed, we can pick an \'etale cover $\cU\to U$ where $\cU$ is a disjoint union of affines (so also $\cU\times_U\cU$ will be affine), and from Grothendiek's theorem \cite[Theorem 11.7]{MR244270} $T(\cU)= r_UT(\cU)$ and $T(\cU\times_U\cU)= r_UT(\cU\times_U\cU)$ as $T$ is smooth. Then from the sheaf properties, the global sections $T(U)$ and $r_U(T|_U)(U)$ are the colimit of the same diagram, so $T(U)=r_U(T|_U)(U)$.
     \item As we can compute $p_*$, $p^{-1}$ and global sections in terms of limits of global sections of $T$ on affine schemes which are \'etale over either $V$ or $U$, proceeding as above $r_V(p^{-1}T|_U)$ and $ p^{-1}_{\text{\'et}}r_U(T|_U)$ are canonically isomorphic and similarly with $p_*$, where we denoted by $p_{\text{\'et}}$ the morphism between the small \'etale sites $V_{\text{\'et}}\to U_{\text{\'et}}$.
     \item For the small \'etale sites we have morphisms
 $r_U(T|_U)(U)\to ((p_{\text{\'et}})_*(p_{\text{\'et}})^{-1})r_U(T|_U)(U)\to r_U(T|_U)(U)$ from \cite[\href{https://stacks.math.columbia.edu/tag/03SH}{Tag 03SH}]{stacks-project}, so we also have morphisms $T(U)\to (p_*p^{-1}T) (U)\to T(U)$ whose composition is multiplication by $d$.
 \end{itemize}

 For part (2) instead, it suffices to check that if one has a cartesian diagram as follows
 \[
 \xymatrix{V' \ar[r]\ar[d]&V\ar[d]^f \\ U'\ar[r]&U}
 \]
with $V\to U$ finite and \'etale, then the trace map of \cite[\href{https://stacks.math.columbia.edu/tag/03SH}{Tag 03SH}]{stacks-project} commutes with base change. Recall that if $\cA$ is a sheaf of abelian groups on the small \'etale site of $U$, the trace map is defined in \cite[\href{https://stacks.math.columbia.edu/tag/03SH}{Tag 03SH}]{stacks-project} as the composition of 
\[\cA\to (f_{\text{\'et}})_*(f_{\text{\'et}})^{-1}\cA\cong (f_{\text{\'et}})_!(f_{\text{\'et}})^{-1}\cA\to \cA\]
where $(f_{\text{\'et}})_*, (f_{\text{\'et}})^{-1}$ and $(f_{\text{\'et}})_!$ are maps between the categories abelian sheaves on the small \'etale site of $U$ and $V$. The two arrows come from natural adjunctions, whereas the isomorphism comes from the map defined in \cite[\href{https://stacks.math.columbia.edu/tag/0F4L}{Tag 0F4L}]{stacks-project}, which is an isomorphism for $f$ finite and \'etale \cite[\href{https://stacks.math.columbia.edu/tag/03S7}{Tag 03S7}]{stacks-project}. Then to check the desired statement, it suffices to observe that both the adjunction maps $r_U(T_U)\to (f_{\text{\'et}})_*(f_{\text{\'et}})^{-1}r_U(T_U)$ and $(f_{\text{\'et}})_!(f_{\text{\'et}})^{-1}r_U(T_U)\to r_U(T_U)$, and the isomorphism $(f_{\text{\'et}})_*(f_{\text{\'et}})^{-1}r_U(T_U)\cong (f_{\text{\'et}})_!(f_{\text{\'et}})^{-1}r_U(T_U)$ commute with base change.

As $f$ is \'etale, the first map $\cA\to (f_{\text{\'et}})_*(f_{\text{\'et}})^{-1}\cA$ evaluated at an \'etale morphism $W\to U$ agrees with the restriction \[\cA(W)\to \cA(W\times_UV).\] The second map instead, up to shrinking $W$, agrees with the addition map \[\bigoplus_{W\to V\text{ over }U}(f_{\text{\'et}})^{-1}\cA(W) = \bigoplus_{W\to V\text{ over }U}\cA(W)\to \cA(W)\]
as $(f_{\text{\'et}})_!$ is the sheafification of $(f_{\text{\'et}})_{p!}$ \cite[\href{https://stacks.math.columbia.edu/tag/03S2}{Tag 03S2}]{stacks-project}.
These commute with base change. Similarly the isomorphism $\cA(W\times_UV)\to \bigoplus_{W\to V\text{ over }U}\cA(W)$ defined in \cite[\href{https://stacks.math.columbia.edu/tag/0F4L}{Tag 0F4L}]{stacks-project}, where we shrank $W$ so that there are $\deg(f)$ sections of $W\to V$ over $U$, again commutes with base change.
\end{proof}
\begin{corollary}\label{cor:Hitorsionrootstack}
    Let $R$ be a henselian DVR , let $\cR$ be a root stack of $\spec R$ at its closed point and $T\to \cR$ a torus. Then $\oH^i(\cR,T)$ is torsion for $i=1,2$.
\end{corollary}
\begin{proof}
 From \Cref{prop_there_is_fet_that_splits_the_torus} there is a finite \'etale morphism $p:\cR'\to \cR$ of degree $d$ such that $T|_{\cR'}\cong \bG_m^n$ for some $n$, with $\cR'$ the root stack of a DVR.
 From \Cref{lemma_p_fet_induces_the_trace_map_on_fppf} the map $$T\to p_*p^{-1}T\to T$$ of abelian fppf-sheaves of $\cX$
is the multiplication by $d$, for $i=1,2$. 
So there are maps whose composition is multiplication by $d$ \[
\oH^{i}(\cR,T)\to \oH^{i}(\cR,p_*p^{-1}T)\to \oH^{i}(\cR,T).
\] It suffices to check that  that $\oH^{i}(\cR,p_*p^{-1}T)\cong \oH^{i}(\cR',p^{-1}T)$. Indeed, $\cR'$ is such that $p^{-1}T=\bG_m^n$ and $\oH^{i}(\cR',\bG_m^n)$ is torsion by \Cref{lemma_torsion_torus_cohom}.
The equality $\oH^{i}(\cR,p_*p^{-1}T)\cong \oH^{i}(\cR',p^{-1}T)$ follows from the Leray spectral sequence, once we prove that $R^ip_*p^{-1}T=0$ for $i>0$; we will now prove this.

From \cite[\href{https://stacks.math.columbia.edu/tag/0GR2}{Tag 0GR2}]{stacks-project} the sheaf $R^ip_*p^{-1}T$ is the sheafification of the presheaf that sends $U\to \cR$ to $\oH^i(U\times_\cR\cR',p^{-1}T)$. In particular to check that $R^ip_*p^{-1}T=0$ it suffices to check that every separated scheme $U\to \cR$ admits an \'
etale cover $\cU\to U$ in affine schemes such that 
$\oH^i(\cU\times_\cR\cR',p^{-1}T)=0$. Let us denote by $V:=\cR'\times_\cR U$ with the second projection $f$. As $f$ is finite, $R^i(f_{\text{\'et}})_*(p^{-1}T)|_{V_{\text{\'et}}}=0$ from \cite[\href{https://stacks.math.columbia.edu/tag/0A4K}{Tag 0A4K}]{stacks-project}, where the subscript \'et stands for the restriction of $f$ and $p^{-1}T$ to the small \'etale sites of $V$ and $U$. 
Then we can find an \'etale cover $\cU\to U$ of $U$ in affine schemes such that $\oH^i_{\text{\'et}}(\cU\times_U V,p^{-1}T|_{V_{\text{\'et}}})=0$.
The vanishing of $R^ip_*p^{-1}T$ follows from Grothendieck's theorem \cite[Theorem 11.7]{MR244270}: \'etale cohomology of affine schemes with value in $p^{-1}T$ (which is smooth) agrees with fppf cohomology, so $\oH^i_{\text{\'et}}(\cU\times_U V,p^{-1}T|_{V_{\text{\'et}}}) = \oH^i(\cU\times_U V,p^{-1}T)$. Then $R^ip_*p^{-1}T=0$ as every scheme $U\to \cR$ admits a cover $\cU\to U$ where $R^ip_*p^{-1}T(\cU)=0$, and $R^ip_*p^{-1}T$ is a sheaf. 
\end{proof}
    
\begin{lemma}\label{lemma_case_BT}
    Let $\cR$ be the root stack of the spectrum of a DVR $R$, let $\eta$ be its generic point and let $T\to \cR$ be a torus. Assume we are given $x_\eta\in \oH^1(\eta,T|_\eta)$. Then, up to replacing $\cR$ with a further root stack, we can extend $x_\eta$ to $x\in \oH^1(\cR,T)$. 
\end{lemma}
\begin{proof}
First observe that $\oH^1(\eta,T|_\eta)$ is torsion. This can be checked as in Corollary \ref{cor:Hitorsionrootstack}, except that in the case of $\eta$ rather than $\cR$ the argument is simpler. Indeed, there is a finite \'etale cover $\eta'\to \eta$ so that $T|_{\eta'}$ is split. Then, using \cite[\href{https://stacks.math.columbia.edu/tag/03SH}{Tag 03SH}]{stacks-project}, there is a sequence
\[\oH^1(\eta,T|_\eta)\to \oH^1(\eta',T|_{\eta'})\cong \oH^1(\eta',\bG_m)\to \oH^1(\eta,T|_\eta)\]
whose composition is the multiplication by a positive integer, which is the degree of $\eta'\to \eta$.
By Lemma \ref{lemma_torsion_torus_cohom} the group $\oH^1(\eta',\bG_m)$ is torsion, so from the sequence above also $\oH^1(\eta,T|_\eta)$ is torsion.  In the fppf topology of $\cR$, we have the following exact sequence:
\[ 1\to F\to T\xrightarrow{\cdot d}T\to 1 
\] where $d$ is the order of $x_\eta$. As $T$ is a torus, $F$ is tame. Then there is a finite tame group $F$ over $\cR$ and an element $c_\eta\in \oH^1(\eta,F|_\eta)$ that maps to $x_\eta$. It suffices to extend $c_\eta$ to $c\in \oH^1(\cR,F)$. Observe that the element $c_\eta$ corresponds to a morphism $\eta \to \cB F$. As $F$ is tame, from \cite{bresciani2024arithmetic} we can extend this morphism to a map $\cR\to \cB F$, up to replacing $\cR$ with a further root stack. This corresponds to an element $c_F\in \oH^1(\cR,F)$ and its image in $\oH^1(\cR,T)$ is the desired extension $x$.
\end{proof}
\begin{proposition}\label{prop_gerbe_banded_by_t_trivial_generically_is_trivial_up_to_further_root}Let $\cR$ be the root stack of a henselian DVR $R$.
    Let $T\to \cR$ be a torus, and $\cT\to \cR$ a gerbe banded by $T$ which is generically trivial. Then, up to replacing $\cR$ with a further root stack, $\cT$ is trivial.
\end{proposition}
\begin{proof}
    Let $c\in \oH^2(\cR,T)$ be the class representing $\cT$. First observe that by Corollary \ref{cor:Hitorsionrootstack}, the class $c$ is torsion.
Let $d$ be the order of $c$, and consider the exact sequence in the fppf-topology
    \[
    1\to F\to T\xrightarrow{\cdot d}T\to 1.
    \]
    As $d$ is the order of $c$, there is $c_F\in \oH^2(\cR,F)$ which maps to $c$ via $\beta:\oH^2(\cR,F)\to \oH^2(\cR,T)$. 
    Consider then the following commutative square, where $\eta$ is the generic point of $\cR$:
    \[
\xymatrix{\oH^1(\cR,T)\ar[r]^\alpha\ar[d] & \oH^2(\cR,F)\ar[r]^\beta\ar[d] & \oH^2(\cR,T)\ar[d]\\\oH^1(\eta,T)\ar[r]^{\alpha|_\eta} & \oH^2(\eta,F)\ar[r]^{\beta|_\eta} & \oH^2(\eta,T).}
    \]
    Since $\beta|_\eta( (c_F)|_\eta)=0$, there is an element $x_\eta\in \oH^1(\eta,T)$ that maps to $(c_F)|_\eta$. From Lemma \ref{lemma_case_BT}, up to replacing $\cR$ with a further root stack,
    we can extend $x_\eta$ to $x\in \oH^1(\cR,T)$. In particular, up to replacing $c_F$ with $c_F - \alpha(x)$, we can assume that $(c_F)|_\eta$ is trivial. So $\cF$, the $F$-gerbe over $\cR$ corresponding to $c_F$, has a section over $\eta$. Thus by \cite{bresciani2024arithmetic} it has a section after replacing $\cR$ with a further root stack. In other terms, $c_F=0$ after passing to a further root stack. Thus also $c=0$.
\end{proof}

\subsection{Proof of Theorem \ref{thm_case_of_gerbe_over_root_stack}}
We are finally ready to prove Theorem \ref{thm_case_of_gerbe_over_root_stack}; we begin with the following auxiliary statement.
\begin{proposition}\label{prop_case_BG}
    Let $R$ be a henselian DVR, let $G\to \cR$ be a reductive group which is either special or has tame Weyl group (as in Definition \ref{def_tame}) where $\cR$ is a root stack of $\spec R$ at its closed point. Let $\eta$ be the generic point of $\cR$. Then any section of $(\cB G)_\eta\to \eta$ can be extended to a section of $\cR\to \cB G$ up to replacing $\cR$ with a further root stack. 
\end{proposition}

\begin{proof} In the special case, there is nothing to prove as $\eta \to \cB G$ classifies the trivial torsor which extends to the trivial torsor over $\cR$. Thus suppose $G$ is reductive with tame Weyl group. Let $T$ be a maximal torus of $G$, which exists by Proposition \ref{prop_there_is_a_max_torus} and let $N$ be the normalizer of $T$ in $G$. The inclusion $N\to G$ induces a map $\oH^1(\cdot, N)\to \oH^1(\cdot, G)$. By \cite[Lemma III.2.2.1 (b)]{serre1979galois} the class $c_\eta$, which represents $\eta\to \cB G$ lifts to an element $d_\eta\in \oH^1(\eta, N)$, as all maximal tori of a reductive group over an algebraically closed field are conjugate.
In particular, it suffices to check that the map $d_\eta:\eta\to \cB N_\eta$ lifts to a $d:\cR\to \cB N$, up to replacing $\cR$ with a further root stack.
Recall that the normalizer of a maximal torus $T$ is an extension of $T$ by the Weyl group which by assumption is tame over $\cR$:
\[
1\to T \to N \to  W\to 1.
\]In particular, we can factor $\cB N\to \cR$ as $\cB N\xrightarrow{\alpha} \cB W \to \cR$. By \cite{bresciani2024arithmetic}, there is a lift of $\alpha\circ d_\eta$ up to replacing $\cR$ with a further root stack, so we have a morphism $\cR\to \cB W$ extending $\alpha\circ d_\eta$. We need now to lift it to $\cB N\times_{\cB W}\cR$, but $\cB N\times_{\cB W}\cR\to \cR$ is a gerbe for a form of the torus $T$ by \cite[\href{https://stacks.math.columbia.edu/tag/0CJY}{Tag 0CJY}]{stacks-project}.
{Let $T'$ be this form of $T$. The gerbe $\cB N \times_{\cB W} \cR\to \cR$ has a section over $\eta$ induced by $d_\eta$, so, as $T'$ is commutative, it is generically trivial. Then \Cref{prop_gerbe_banded_by_t_trivial_generically_is_trivial_up_to_further_root} applies so that after replacing $\cR$ by a further root stack, $\cB N \times_{\cB W} \cR\cong \cB T'$. Under this isomorphism, the section over $\eta$ induced by $d_\eta$ corresponds to a $T'|_\eta$-torsor, which by \Cref{lemma_case_BT} extends to a $T'$-torsor over $\cR$ after replacing $\cR$ with a further root stack. This gives a section $\cR \to \cB N \times_{\cB W} \cR$ extending $\eta \to \cB N \times_{\cB W} \cR$, as required.}
\end{proof}

\begin{remark}\label{rem:bounding} In the special case of a gerbe over $R$ itself rather than a root stack of $\spec R$, the order of the root stack needed to extend the $K$-point of $\cB G$ can be bounded uniformly in terms of the Weyl group of $G$. Indeed by work of Chernousov, Gille and Reichstein \cite{chernousov2006resolving, chernousov2008reduction} there exists a finite subgroup scheme $S \subset G$ which is an extension of the Weyl group of $G$ by a diagonalizable group such that $\oH^1(K,S) \to \oH^1(K,G)$ is surjective. Thus we can lift the $K$-point of $\cB G$ to $\cB S$. By the tameness assumption on the Weyl group of $G$, the stack $\cB S$ is tame so we may apply \cite{bresciani2024arithmetic} to this stack and then the degree of the root stack necessary is bounded by the order of $S$. 
    
\end{remark}

We are now ready to prove Theorem \ref{thm_case_of_gerbe_over_root_stack}.

\begin{proof}[Proof of Theorem \ref{thm_case_of_gerbe_over_root_stack}]
From Proposition \ref{prop_case_BG}, it suffices to show that, up to replacing $\cR$ with a further root stack, we have that
{$\cG \cong \cB G'$ for an inner form $G'$ of $G$ over $\cR$. Indeed, $G'$ is reductive with tame Weyl group, as its geometric fibers are isomorphic to those of $G$, so \Cref{prop_case_BG} applies to the $K$-point of $\cB G'$ corresponding to $\phi$.} Since $\cG$ is banded by $G$, it corresponds to a class $c\in \oH^2(\cR,G)$. Let $Z$ be the center of $G$ (which exists by \cite[Lemma 2.2.4]{ConradReductive}). Observe that $Z$ is linearly reductive as $Z\subseteq C(T)$ where $C(T)$ is the centralizer of a maximal torus $T$ of $G$ (which again exists by \cite[Lemma 2.2.4]{ConradReductive}), and $C(T)=T$ as it is true on geometric fibers and both sides commute with base change by \cite[Theorem 3.3.4]{ConradReductive}. 

{The $K$-point $\phi$ only shows that $\cG|_\eta$ is neutral: $\cG|_\eta\cong \cB G'_\eta$ where $G'_\eta:=\Aut(\phi)$, and $G'_\eta$ need not be isomorphic to $G|_\eta$. Since $\cG$ is banded by $G$, the sheaf $P$ of isomorphisms $G|_\eta\to G'_\eta$ inducing the identification of bands given by $\cG$ is a torsor under $G^{ad}:=G/Z$, acting by precomposition with inner automorphisms, and $G'_\eta\cong P\times^{G^{ad}}G|_\eta$, where $G^{ad}$ acts on $G$ by conjugation. The group $G^{ad}$ is reductive and its geometric fibers have the same Weyl groups as those of $G$, so it has tame Weyl group. Thus by \Cref{prop_case_BG} applied to $G^{ad}$, up to replacing $\cR$ with a further root stack, $P$ extends to a $G^{ad}$-torsor $Q$ over $\cR$. Let $G_Q:=Q\times^{G^{ad}}G$ be the corresponding inner form of $G$ over $\cR$. Then $\cB G_Q$ is banded by $G$, the center of $G_Q$ is $Z$, and $\cB G_Q|_\eta\cong \cB G'_\eta\cong \cG|_\eta$ compatibly with the bands.}

It follows from \cite[Thm. IV.3.3.3]{giraud} (see also the proof of \cite[Proposition 2.1]{bejleri2024proper}) that the map $\oH^2(\cR,Z)\to \oH^2(\cR,G)$ and the map $\oH^2(\eta,Z)\to \oH^2(\eta,G)$ defined by sending a gerbe $\cZ$ to $(\cZ\times \cB G_{{Q}})\thickslash Z$ are bijections, where $(\cZ\times \cB G_{{Q}})\thickslash Z$ is the rigidification of $\cZ\times \cB G_{{Q}}$ by the diagonal subgroup $Z$ of the inertia stack $\cI (\cZ \times \cB G_{{Q}}) = \cI \cZ \times \cI \cB G_{{Q}}$. {Let $\cZ$ be the $Z$-gerbe corresponding to $\cG$. As these bijections are compatible with restriction to $\eta$ and $\cG|_\eta\cong \cB G_Q|_\eta$, the gerbe $\cZ$ is generically trivial; moreover, if $\cZ$ is trivial then $\cG\cong \cB G_Q$.}
In particular, it suffices to show that if $Z$ is a commutative linearly reductive group over $\cR$, and $\cZ\to \cR$ is a gerbe banded by $Z$ which is generically trivial, it can be trivialized up to replacing $\cR$ with a further root stack.
We can further rigidify by the smooth connected component of the identity $(Z^0)_{sm}$ (\cite[Theorem 9.9]{AHRetalelocal} and \cite[Remark 2.2]{bejleri2024proper}) and factor $\cZ\to \cR$ as $\cZ\to \cZ'\to \cR$ where $\cZ\to \cZ'$ is a gerbe for $(Z^0)_{sm}$, and $\cZ'\to \cR$ is finite and tame \cite[Proposition 2.3]{bejleri2024proper}. Applying \cite{bresciani2024arithmetic} to the generically trivial tame gerbe $\cZ'\to \cR$, after replacing $\cR$ by a further root stack, we can assume that $\cZ' \to \cR$ has a section. Pulling back $\cZ \to \cZ'$ along this section, we obtain a gerbe for $(Z_0)_{sm}$ so it suffices to treat the case in which $(Z^0)_{sm} = Z$ is a smooth commutative connected reductive group, i.e. a torus. Then Proposition \ref{prop_gerbe_banded_by_t_trivial_generically_is_trivial_up_to_further_root} applies so that $\cZ$
{is trivialized, and thus $\cG\cong \cB G_Q$,} after replacing $\cR$ by a further root stack. 
 \end{proof}

\section{Proofs of the main theorems and applications}

We are finally ready to prove \Cref{thm:main} and \Cref{thm:main2}.
\begin{proof}[Proof of \Cref{thm:main}]
We first assume that $R$ is henselian. From
\Cref{prop_knowing_the_result_for_poli_pts_suffices} it suffices to treat the case when the generic point of $\spec R$ maps to a polystable point of $\cX$. From \Cref{cor_reduction_to_gerbe} it suffices to assume that $\cX\to X$ is a gerbe banded by a connected, smooth, linearly reductive group. This case is treated in \Cref{thm_case_of_gerbe_over_root_stack}.
To reduce to the henselian case we use \cite[Theorem B]{rydh2011etale}. Take $\spec A\to \spec R$ the Henselization morphism. We know that the n-th root stack of $\spec A$ at its closed point will admit a lifting to $\cX$. But the following diagram is a pushout, which concludes the argument

\[
\xymatrix{\spec \Frac(A)\ar[r] \ar[d] & \sqrt[n]{\spec A}\ar[d] \\ \spec \Frac(R) \ar[r] &\sqrt[n]{\spec R}.}
\]
\end{proof}

\begin{proof}[Proof of \Cref{thm:main2}]

When $R$ is henselian, this is exactly \Cref{thm_case_of_gerbe_over_root_stack}. To reduce to the henselian case, we can use the same pushout diagram as in the above proof of \Cref{thm:main}. 
\end{proof}

\begin{proof}[Proof of \Cref{thm_mixed_char_adequate}] For the root stack valuative criterion, we may replace $\cX$ with $\cX_{red}$ and assume $\cX$ and $X$ are both reduced. For any integer $m$, we denote by $\cX[\frac{1}{m}] \to X[\frac{1}{m}]$ the pullback of $\cX \to X$ to $\spec \Z[\frac{1}{m}]$. Let $\cX_{\Q} \to X_\Q$ be the generic fiber and let $\cZ_{\Q} \to \cX_{\Q}$ be the stack constructed in \Cref{thm_reduction_to_gerbe_general_case}. We can spread out $\cZ_\Q$ to some stack $\cZ$ over $\spec \Z[\frac{1}{m}]$ for some $m$, and up to increasing $m$, we can assume we have a morphism $\cZ \to \cX[\frac{1}{m}]$ extending the $\cZ_{\Q} \to \cX_{\Q}$ and $\cZ \to \cY \to  Z$ extending the factorization of the good moduli space map $\cZ_{\Q} \to Z_\Q$ in \Cref{thm_reduction_to_gerbe_general_case}(3).  

Up to increasing $m$, we may assume that
\begin{enumerate}
    \item the algebraic stack $\cY$ is tame,
    \item the map $\cZ\to \cY$ is a gerbe with geometric fibers $\cB G$ with $G$ a reductive (so also smooth and connected) group, and
    \item and all the groups $G$ appearing in the previous point have tame Weyl group.
\end{enumerate}

Now let $\spec R \to X$ and $x_K: \spec K \to \cX$ be given where $R$ is a DVR of mixed characteristic $(0,p)$ with $p \nmid m$. Since $K$ has characteristic $0$, $x_K$ factors through $\cX_\Q$. Similarly, $\spec R \to X$ factors through $X[\frac{1}{m}]$. Now we may apply \Cref{prop_knowing_the_result_for_poli_pts_suffices} to reduce the theorem to the case that $x_K$ is a polystable $K$-point so without loss of generality, we assume $x_K$ is polystable. 

By \Cref{thm_reduction_to_gerbe_general_case}(2), $x_K$ lifts to a stable $K$-point of $\cZ_\Q$. By composition, we obtain $K$-points of $\cZ$ and $\cY$. Since $\cY$ is tame, we can apply \cite[Theorem 3.1]{bresciani2024arithmetic} to obtain a root stack $\cR \to R$ and a lift $\cR \to \cY$. This situation is summarized in the following diagram. 
\[
\xymatrix{
 & \cG \ar[r] \ar[d] \ar@{}[dr]|(.2){\big\lrcorner}  & \cZ\ar[dr] \ar[d]_{\operatorname{gerbe}} & \\
 & \cR\ar[d] \ar[r] & \cY \ar[dr]_{\operatorname{tame}} & \cX[\frac{1}{m}]\ar[d]\\
 \spec K\ar[uur]\ar[r] & \spec R  \ar[rr] & & X[\frac{1}{m}].
}
\]

As in the proof of \Cref{thm:main}, we may assume that $R$ is henselian using \cite[Theorem B]{rydh2011etale}. Moreover, by construction, the pullback $\cG = \cZ \times_{\cY} \cR \to \cR$ is a gerbe banded by a connected, smooth, reductive group with tame Weyl group. This case is treated in \Cref{thm_case_of_gerbe_over_root_stack}, as before.
\end{proof}

\begin{proof}[Proof of \Cref{prop:main_global}, \Cref{cor_global} and \Cref{cor:number_fields}]

For \Cref{prop:main_global}, let $x_1, \ldots, x_m \in C$ be points of indeterminacy for $\varphi$ and let $U$ be the open complement so that $\varphi$ restricts to a morphism $U \to \cX$. Let $R_i = \cO_{C,x_i}$ which is a DVR and let $K = k(C) = \operatorname{Frac}R_i$ be the function field. The $K$-point $\spec K \to \cX$ extends to root stack points $\sqrt[n_i]{\spec R_i} \to \cX$ by properness of $X$ and \Cref{thm:main} or by \Cref{thm:main2} for some integers $n_i$. Let $\cC \to C$ be the global root stack $\sqrt[n_1, \ldots, n_m]{(C, x_1, \ldots, x_m)}$. We wish to extend $U \to \cX$ to a morphism $\cC \to \cX$. By induction on $m$, we can extend the morphism one point at a time so suppose without loss of generality that $m = 1$.  Consider the diagram 
$$
\xymatrix{\spec K \ar[r] \ar[d] & \sqrt[n]{\spec R} \ \ar[d] \\ U \ar[r] & \cC}.
$$
This is a flat Mayer-Vietoris square in the sense of \cite[Definition 1.2]{MV} and $\cC$ is locally excellent by assumption so this square is a pushout square in algebraic stacks by \cite[Theorem A]{MV}. Thus there is a morphism $\cC \to \cX$ by the universal property of pushouts. 

For \Cref{cor_global}, note that $\cX$ is locally of finite presentation by assumption so the given $K$-point $\spec K \to \cX$ spreads out to a morphism $U \to \cX$ for some open set $U \subset C$ \cite[\href{https://stacks.math.columbia.edu/tag/0CMX}{Tag 0CMX}]{stacks-project}. The result then follows from \Cref{prop:main_global}. 

For \Cref{cor:number_fields}, we again spread out to some $U \to \cX[\frac{1}{m}]$ where $U = \spec O_K[\frac{1}{m}]$ and $m$ is chosen so that $X[\frac{1}{m}]$ is proper and $m$ satisfies the conclusion of \Cref{thm_mixed_char_adequate}. Then for every $x \in \spec O_K[\frac{1}{m}] \setminus U$, the $K$-point extends to a map $\spec \cO_{K,x} \to X[\frac{1}{m}]$ by properness and then lifts to a root stack point $\sqrt[n]{\spec \cO_{K,x}} \to \cX[\frac{1}{m}]$ by \Cref{thm_mixed_char_adequate}. Now we proceed exactly as in the case of Proposition \ref{prop:main_global}.
\end{proof}

\begin{proof}[Proof of \Cref{cor:torsors}] By \Cref{thm:main2} and \Cref{rem:bounding}, there exists an $n$ depending only on $G$ such that any $K$-point extends to a $\sqrt[n]{\spec R}$ point of $\cB G$. 
\end{proof}

\begin{proof}[Proof of \Cref{prop_LN} and \Cref{cor:LN_1}] Exactly the same argument as in \cite[Theorem 4.1]{bresciani2024arithmetic} goes through, except we replace \cite[Corollary 3.2]{bresciani2024arithmetic} with Theorem \ref{thm:main} which gives a morphism $\cR\to \cY$ that, if pre-composed with the covering $\spec(R)\to \cR$ and the closed embedding of the closed point $\spec(k_R)\to \spec(R)$, gives the desired point on $\cY$. 

For \Cref{cor:LN_1}, we can apply \Cref{prop_LN} to the rational map $\spec R \dashrightarrow [\spec A/G]$ to get a $k$-point of $[\spec A\otimes_R k/G_k]$ which lifts to a $k$-point of $A\otimes_R k$ by the special assumption.  
\end{proof} 

\begin{proof}[Proof of \Cref{prop_gs_tame}]
Let $\cG \in \oH^2(R,G)$ such that $\cG|_K = 0$, i.e.\ $\cG|_K\cong \cB G_K$ as gerbes banded by $G$. Let $Z$ be the center of $G$, and let $\cZ$ be the $Z$-gerbe corresponding to $\cG$ under the bijection $\oH^2(-,Z)\to \oH^2(-,G)$, $\cZ\mapsto (\cZ\times \cB G)\thickslash Z$, of \cite[Thm. IV.3.3.3]{giraud}, as in the proof of \Cref{thm_case_of_gerbe_over_root_stack}. Since $\cG|_K\cong \cB G_K$, the gerbe $\cZ$ is generically trivial, and it suffices to show that $\cZ$ becomes trivial after passing to a root stack. Fix a section $s_K$ of $\cZ|_K$; we show that, after passing to a root stack $\cR_n$, it extends to a section of $\cZ$. If $R$ is henselian, this follows as in the last part of the proof of \Cref{thm_case_of_gerbe_over_root_stack}, by rigidifying, applying \cite{bresciani2024arithmetic} and  \Cref{prop_gerbe_banded_by_t_trivial_generically_is_trivial_up_to_further_root} as in \Cref{thm_case_of_gerbe_over_root_stack}. The general case follows from the henselian one using the same pushout diagram as in the proof of \Cref{thm:main}. Finally, $\spec k\to \spec R$ factors through $\cR_n$: also $\cG|_k$ is trivial.
\end{proof}

\begin{proof}[Proof of \Cref{cor:homogeneous_1} and \Cref{cor:unirationality}]
{Let $x_K : \spec K \to [V/H]$ be the composition of a $K$-point of $V$ with the quotient map. The stack $[V/H]\to \spec R$ is a gerbe whose geometric stabilizers are reductive with tame Weyl group, so by \Cref{prop_every_gerbe_is_banded} it is banded by a reductive group over $\spec R$ with tame Weyl group. Thus by \Cref{thm:main2}, $x_K$ extends to a section $\cR_n\to [V/H]$ for some root stack $\cR_n$ of $\spec R$. Composing with $\spec k\to \cR_n$ gives a $k$-point of $[V/H]$, which lifts to a $k$-point of $V$ since $H$ is special.}

If $V_K$ is unirational then it admits a $K$-point. Thus $V_k$ also admits a $k$-point. By \cite[Theorem 18.2(ii)]{borel}, $H_k$ is unirational and given a $k$-point $x \in V(k)$, we get a surjection $H_k \to V_k$ by $h \mapsto hx$ which implies $V_k$ is unirational. 
\end{proof}

\begin{remark} In fact Theorem \ref{thm:main2} applied to $[V/H]$ shows something slightly stronger. Given any $K$-point $x_K \in V(K)$, there exists a root extension $R'=R[t^{1/n}]$ where $t$ is a uniformizer and an element $h \in H(R')$ such that $hx_K$ extends to an $R'$-point of $V$. 
\end{remark}

\begin{proof}[Proof of \Cref{cor:K_moduli} and \Cref{cor:bpcy_moduli}]

The $K$-moduli stack $\cM^{\text{K-ss}}$ of $K$-semistable Fano varieties has a proper good moduli space $M^{\text{K-ps}}$ by \cite{smoothable_Kmoduli, open_Kmoduli, gms_kmoduli, liu2022finite}. Thus by \Cref{thm:main}, there exists a root stack $\cR_n$ and a family $\cY \to \cR_n$ of $K$-semistable Fano varieties with polystable central fiber extending $Y_\eta \to \spec K$. The total space has klt singularities by \cite[Theorem 1.3]{odaka} and inversion of adjunction for klt singularities \cite[Sections 16 \& 17]{flipsandabundance} and quotients of klt singularities are klt
%\Gio{Possible alternative with better ref "If we denote by $\cY_0$ the central fiber of $\cY\to \cR_n$, then by the definition of (locally) stable family, the pair $(\cY,\cY_0)$ has log-canonical singularities, and the only lc centers are over the special fiber of $\cY$ (recall that a Deligne-Mumford stack in characteristic 0 has log-canonical singularities if an \'etale cover of it does). If we denote by $\cY\to Y$ the coarse moduli space map and by $Y_0$ the central fiber of $Y\to \spec R$ with the reduced structure, then one can check that the morphism $(\cY,\cY_0)\to (Y,Y_0)$ is crepant. So from \cite[\S 2.3]{kollar2013singularities}, also the pair $(Y,Y_0)$ has lc singularities, and the lc centers map to $Y_0$. Therefore the pair $(Y,0)$ is klt."} \cite{braun2024reductive, zhuang2024direct}
so the coarse space $Y \to \spec R$ has klt singularities and central fiber which is a quotient of a $K$-polystable Fano variety by $\bmu_n$. The multiplicity $n$ can be deduced by noting that $\cR_n \to \spec R$ is totally ramified of order $n$ (see e.g. \cite[Lemma 7.11]{bejleri2024heightmodulicyclotomicstacks}).

For the moduli of boundary polarized Calabi-Yau surface pairs $\cM^{CY}$, the argument is similar, except here we only have an \emph{asymptotically good moduli space} \cite[Definition 13.1]{ABBDILW}. More precisely, we bound the Cartier index of $K_Y$ to obtain an open substack $\cM^{CY}_m$ which has a good moduli space $M^{CY}_m$. By the proofs of \cite[Theorem 1.3]{ABBDILW} and \cite[Theorem 1.1]{blum2024goodmodulispacesboundary}, the good moduli space is proper and independent of $m$ for $m \gg 0$. Thus we apply \Cref{thm:main2} to $\cM^{CY}_m$ for $m$ sufficiently large. Then the proof is exactly as above using that quotients of slc singularities are slc \cite[Lemma 2.3]{alexeev2012non} and inversion of adjunction for log canonical singularities \cite[Main Theorem]{kawakita}.
\end{proof}

\bibliographystyle{alpha}
\bibliography{gms-tame-root-stack}

@article{gille2024loop,
  title={Loop torsors. Theory and applications},
  author={Gille, Philippe and Chernousov, Vladimir and Pianzola, Arturo},
  journal={arXiv preprint arXiv:2412.07294},
  year={2024}
}

@misc{achenjang,
  bibtex_show = {true},
  title = {On Brauer groups of tame stacks},
  author = {Achenjang, Niven},
  year = {2024},
  eprint = {2410.06217},
  archiveprefix = {arXiv},
  primaryclass = {math.AG},
  url = {https://arxiv.org/abs/2410.06217},
  selected = {true},
  arxiv = {2410.06217},
  journal = {arxiv}
}

@misc{sheng2024nonabelianhodgecorrespondenceprincipal,
      title={A Nonabelian Hodge Correspondence for Principal Bundles in Positive Characteristic}, 
      author={Mao Sheng and Hao Sun and Jianping Wang},
      year={2024},
      eprint={2405.09947},
      archivePrefix={arXiv},
      primaryClass={math.AG},
      url={https://arxiv.org/abs/2405.09947}, 
}

@article {pappas_rapoport,
    AUTHOR = {Pappas, Georgios and Rapoport, Michael},
     TITLE = {On tamely ramified {$\mathcal{G}$}-bundles on curves},
      NOTE = {With an appendix by Brian Conrad},
   JOURNAL = {Algebr. Geom.},
  FJOURNAL = {Algebraic Geometry},
    VOLUME = {11},
      YEAR = {2024},
    NUMBER = {6},
     PAGES = {796--829},
      ISSN = {2313-1691,2214-2584},
   MRCLASS = {14L15 (14D23 20G35 20J06)},
  MRNUMBER = {4819392},
}

@article {balaji_seshadri,
    AUTHOR = {Balaji, V. and Seshadri, C. S.},
     TITLE = {Moduli of parahoric {$\mathscr{G}$}-torsors on a compact {R}iemann
              surface},
   JOURNAL = {J. Algebraic Geom.},
  FJOURNAL = {Journal of Algebraic Geometry},
    VOLUME = {24},
      YEAR = {2015},
    NUMBER = {1},
     PAGES = {1--49},
      ISSN = {1056-3911,1534-7486},
   MRCLASS = {14D22 (14H10 20H10)},
  MRNUMBER = {3275653},
MRREVIEWER = {Usha\ N.\ Bhosle},
       DOI = {10.1090/S1056-3911-2014-00626-3},
       URL = {https://doi.org/10.1090/S1056-3911-2014-00626-3},
}

@book{olsson2016algebraic,
  title={Algebraic spaces and stacks},
  author={Olsson, Martin},
  volume={62},
  year={2016},
  publisher={American Mathematical Soc.}
}

@misc{martens2015,
      title={Variations on a theme of {G}rothendieck}, 
      author={Johan Martens and Michael Thaddeus},
      year={2015},
      eprint={1210.8161},
      archivePrefix={arXiv},
      primaryClass={math.AG},
      url={https://arxiv.org/abs/1210.8161}, 
}

@book {serre1979galois,
    AUTHOR = {Serre, Jean-Pierre},
     TITLE = {Galois cohomology},
      NOTE = {Translated from the French by Patrick Ion and revised by the
              author},
 PUBLISHER = {Springer-Verlag, Berlin},
      YEAR = {1997},
     PAGES = {x+210},
      ISBN = {3-540-61990-9},
   MRCLASS = {12G05 (11R34)},
  MRNUMBER = {1466966},
       DOI = {10.1007/978-3-642-59141-9},
       URL = {https://doi-org.offcampus.lib.washington.edu/10.1007/978-3-642-59141-9},
}

@article {AbramovichVistoli,
    AUTHOR = {Abramovich, Dan and Vistoli, Angelo},
     TITLE = {Compactifying the space of stable maps},
   JOURNAL = {J. Amer. Math. Soc.},
  FJOURNAL = {Journal of the American Mathematical Society},
    VOLUME = {15},
      YEAR = {2002},
    NUMBER = {1},
     PAGES = {27--75},
      ISSN = {0894-0347,1088-6834},
   MRCLASS = {14H10 (14A20 14D20 14N35)},
  MRNUMBER = {1862797},
MRREVIEWER = {Tyler\ J.\ Jarvis},
       DOI = {10.1090/S0894-0347-01-00380-0},
       URL = {https://doi.org/10.1090/S0894-0347-01-00380-0},
}

@article {Alper,
    AUTHOR = {Alper, Jarod},
     TITLE = {Good moduli spaces for {A}rtin stacks},
   JOURNAL = {Ann. Inst. Fourier (Grenoble)},
  FJOURNAL = {Universit\'{e} de Grenoble. Annales de l'Institut Fourier},
    VOLUME = {63},
      YEAR = {2013},
    NUMBER = {6},
     PAGES = {2349--2402},
      ISSN = {0373-0956},
   MRCLASS = {14D23 (14L24 14L30)},
  MRNUMBER = {3237451},
MRREVIEWER = {Arvid Siqveland},
       URL = {http://aif.cedram.org/item?id=AIF_2013__63_6_2349_0},
}

@inproceedings{abramovich2011stable,
  title={Stable varieties with a twist},
  author={Abramovich, Dan and Hassett, Brendan},
  booktitle={Classification of algebraic varieties},
  pages={1--38},
  year={2011}
}

@ARTICLE{AHRetalelocal,
       author = {Alper, Jarod and Hall, Jack and Rydh, David},
        title = "{The \'etale local structure of algebraic stacks}",
      journal = {arXiv e-prints},
         year = 2025,
        month = apr,
          eid = {https://arxiv.org/pdf/1912.06162.pdf},
        pages = {https://arxiv.org/pdf/1912.06162.pdf},
archivePrefix = {arXiv},
       eprint = {https://arxiv.org/pdf/1912.06162.pdf},
 primaryClass = {math.AG},
}

@article{AHLH,
  title={Existence of moduli spaces for algebraic stacks},
  author={Alper, Jarod and Halpern-Leistner, Daniel and Heinloth, Jochen},
  journal={Inventiones mathematicae},
  volume={234},
  number={3},
  pages={949--1038},
  year={2023},
  publisher={Springer}
}

@article {Alperadequate,
    AUTHOR = {Alper, Jarod},
     TITLE = {Adequate moduli spaces and geometrically reductive group
              schemes},
   JOURNAL = {Algebr. Geom.},
  FJOURNAL = {Algebraic Geometry},
    VOLUME = {1},
      YEAR = {2014},
    NUMBER = {4},
     PAGES = {489--531},
      ISSN = {2313-1691,2214-2584},
   MRCLASS = {14D20 (14A20 14L15 14L24 14L30)},
  MRNUMBER = {3272912},
MRREVIEWER = {Fabio\ Tonini},
       DOI = {10.14231/AG-2014-022},
       URL = {https://doi.org/10.14231/AG-2014-022},
}

@ARTICLE{ABBDILW,
       author = {Ascher, Kenneth and Bejleri, Dori and Blum, Harold and DeVleming, Kristin and Inchiostro, Giovanni and Liu, Yuchen and Wang, Xiaowei},
        title = "{{M}oduli of boundary polarized {C}alabi-{Y}au pairs}",
      journal = {arXiv e-prints},
         year = 2023,
        month = jul,
          eid = {arXiv:2307.06522},
        pages = {arXiv:2307.06522},
archivePrefix = {arXiv},
       eprint = {2307.06522},
 primaryClass = {math.AG},
}

@misc{blum2024goodmodulispacesboundary,
      title={Good moduli spaces for boundary polarized Calabi-Yau surface pairs}, 
      author={Harold Blum and Yuchen Liu},
      year={2024},
      eprint={2407.00850},
      archivePrefix={arXiv},
      primaryClass={math.AG},
      url={https://arxiv.org/abs/2407.00850}, 
}

@misc{bejleri_jeon,
    title = {Parahoric extension for reductive torsors},
    author = {Dori Bejleri and Myeong Jae Jeon},
    note = {In preparation},
    }

@misc{gwz2,
    title = {$\chi$-independence for K3-surfaces via $p$-adic integration}, 
    author = {Michael Groechenig and Dimitri Wyss and Paul Ziegler}, 
    note = {In preparation},
    }

@misc{groechenig2025twistedpointsquotientstacks,
      title={Twisted points of quotient stacks, integration and BPS-invariants}, 
      author={Michael Groechenig and Dimitri Wyss and Paul Ziegler},
      year={2025},
      eprint={2409.17358},
      archivePrefix={arXiv},
      primaryClass={math.AG},
      url={https://arxiv.org/abs/2409.17358}, 
}

@article {ed_stacks,
    AUTHOR = {Brosnan, Patrick and Reichstein, Zinovy and Vistoli, Angelo},
     TITLE = {Essential dimension of moduli of curves and other algebraic
              stacks},
      NOTE = {With an appendix by Najmuddin Fakhruddin},
   JOURNAL = {J. Eur. Math. Soc. (JEMS)},
  FJOURNAL = {Journal of the European Mathematical Society (JEMS)},
    VOLUME = {13},
      YEAR = {2011},
    NUMBER = {4},
     PAGES = {1079--1112},
      ISSN = {1435-9855,1435-9863},
   MRCLASS = {14D23 (14H10 14K10)},
  MRNUMBER = {2800485},
MRREVIEWER = {Jack\ Hall},
       DOI = {10.4171/JEMS/276},
       URL = {https://doi.org/10.4171/JEMS/276},
}

@article {BV2,
    AUTHOR = {Bresciani, Giulio and Vistoli, Angelo},
     TITLE = {Fields of moduli and the arithmetic of tame quotient
              singularities},
   JOURNAL = {Compos. Math.},
  FJOURNAL = {Compositio Mathematica},
    VOLUME = {160},
      YEAR = {2024},
    NUMBER = {5},
     PAGES = {982--1003},
      ISSN = {0010-437X,1570-5846},
   MRCLASS = {14D10 (14D22 14J20)},
  MRNUMBER = {4723349},
       DOI = {10.1112/S0010437X2400705X},
       URL = {https://doi.org/10.1112/S0010437X2400705X},
}

@article {AOV,
    AUTHOR = {Abramovich, Dan and Olsson, Martin and Vistoli, Angelo},
     TITLE = {Twisted stable maps to tame {A}rtin stacks},
   JOURNAL = {J. Algebraic Geom.},
  FJOURNAL = {Journal of Algebraic Geometry},
    VOLUME = {20},
      YEAR = {2011},
    NUMBER = {3},
     PAGES = {399--477},
      ISSN = {1056-3911,1534-7486},
   MRCLASS = {14D23},
  MRNUMBER = {2786662},
MRREVIEWER = {Hsian-Hua\ Tseng},
       DOI = {10.1090/S1056-3911-2010-00569-3},
       URL = {https://doi.org/10.1090/S1056-3911-2010-00569-3},
}

@article {MV,
    AUTHOR = {Hall, Jack and Rydh, David},
     TITLE = {Mayer-{V}ietoris squares in algebraic geometry},
   JOURNAL = {J. Lond. Math. Soc. (2)},
  FJOURNAL = {Journal of the London Mathematical Society. Second Series},
    VOLUME = {107},
      YEAR = {2023},
    NUMBER = {5},
     PAGES = {1583--1612},
      ISSN = {0024-6107,1469-7750},
   MRCLASS = {14A20 (13F40 14F08 14F20)},
  MRNUMBER = {4585296},
MRREVIEWER = {Ariyan\ Javanpeykar},
       DOI = {10.1112/jlms.12575},
       URL = {https://doi.org/10.1112/jlms.12575},
}

@book {borel,
    AUTHOR = {Borel, Armand},
     TITLE = {Linear algebraic groups},
    SERIES = {Graduate Texts in Mathematics},
    VOLUME = {126},
   EDITION = {Second},
 PUBLISHER = {Springer-Verlag, New York},
      YEAR = {1991},
     PAGES = {xii+288},
      ISBN = {0-387-97370-2},
   MRCLASS = {20-01 (20Gxx)},
  MRNUMBER = {1102012},
MRREVIEWER = {F.\ D.\ Veldkamp},
       DOI = {10.1007/978-1-4612-0941-6},
       URL = {https://doi.org/10.1007/978-1-4612-0941-6},
}

@article{edidin2001brauer,
  title={Brauer groups and quotient stacks},
  author={Edidin, Dan and Hassett, Brendan and Kresch, Andrew and Vistoli, Angelo},
  journal={American Journal of Mathematics},
  volume={123},
  number={4},
  pages={761--777},
  year={2001},
  publisher={Johns Hopkins University Press}
}

@incollection {MR244270,
    AUTHOR = {Grothendieck, Alexander},
     TITLE = {Le groupe de {B}rauer. {II}. {T}h\'eorie cohomologique},
 BOOKTITLE = {Dix expos\'es sur la cohomologie des sch\'emas},
    SERIES = {Adv. Stud. Pure Math.},
    VOLUME = {3},
     PAGES = {67--87},
 PUBLISHER = {North-Holland, Amsterdam},
      YEAR = {1968},
   MRCLASS = {14F22},
  MRNUMBER = {244270},
MRREVIEWER = {James\ Milne},
}

@article{rydh2011etale,
  title={{\'E}tale d{\'e}vissage, descent and pushouts of stacks},
  author={Rydh, David},
  journal={Journal of Algebra},
  volume={331},
  number={1},
  pages={194--223},
  year={2011},
  publisher={Elsevier}
}

@article{bate2017cocharacter,
  title={Cocharacter-closure and the rational {H}ilbert--{M}umford theorem},
  author={Bate, Michael and Herpel, Sebastian and Martin, Benjamin and R{\"o}hrle, Gerhard},
  journal={Mathematische Zeitschrift},
  volume={287},
  pages={39--72},
  year={2017},
  publisher={Springer}
}

@article{serre2010bounds,
  title={Bounds for the orders of the finite subgroups of {$G$}(k)},
  author={Serre, Jean-Pierre},
  journal={arXiv preprint arXiv:1011.0346},
  year={2010}
}

@article{lieblich2008twisted,
  title={Twisted sheaves and the period-index problem},
  author={Lieblich, Max},
  journal={Compositio Mathematica},
  volume={144},
  number={1},
  pages={1--31},
  year={2008},
  publisher={London Mathematical Society}
}

@inproceedings{alper2024artin,
  title={Artin algebraization for pairs with applications to the local structure of stacks and Ferrand pushouts},
  author={Alper, Jarod and Hall, Jack and Halpern-Leistner, Daniel and Rydh, David},
  booktitle={Forum of Mathematics, Sigma},
  volume={12},
  pages={e20},
  year={2024},
  organization={Cambridge University Press}
}

@article{kempf1978instability,
  title={Instability in invariant theory},
  author={Kempf, George R},
  journal={Annals of Mathematics},
  volume={108},
  number={2},
  pages={299--316},
  year={1978},
  publisher={JSTOR}
}

@article {Douai,
    AUTHOR = {Douai, Jean-Claude},
     TITLE = {{$2$}-cohomologie galoisienne des groupes semi-simples
              d\'{e}finis sur les corps locaux},
   JOURNAL = {C. R. Acad. Sci. Paris S\'{e}r. A-B},
  FJOURNAL = {Comptes Rendus Hebdomadaires des S\'{e}ances de l'Acad\'{e}mie
              des Sciences. S\'{e}ries A et B},
    VOLUME = {280},
      YEAR = {1975},
    NUMBER = {6},
     PAGES = {Aii, A321--A323},
      ISSN = {0151-0509},
   MRCLASS = {12B20 (14F99 14L25 20G25)},
  MRNUMBER = {401713},
MRREVIEWER = {Michiel\ Hazewinkel},
}

@article {giraud,
    AUTHOR = {Giraud, Jean},
     TITLE = {Cohomologie non ab\'{e}lienne; pr\'{e}liminaires},
   JOURNAL = {C. R. Acad. Sci. Paris},
  FJOURNAL = {Comptes Rendus Hebdomadaires des S\'{e}ances de l'Acad\'{e}mie
              des Sciences},
    VOLUME = {260},
      YEAR = {1965},
     PAGES = {2392--2394},
      ISSN = {0001-4036},
   MRCLASS = {18.20 (14.55)},
  MRNUMBER = {201489},
}

@article{bresciani2024arithmetic,
  title={An arithmetic valuative criterion for proper maps of tame algebraic stacks},
  author={Bresciani, Giulio and Vistoli, Angelo},
  journal={manuscripta mathematica},
  volume={173},
  number={3},
  pages={1061--1071},
  year={2024},
  publisher={Springer}
}

@article {gms_kmoduli,
    AUTHOR = {Alper, Jarod and Blum, Harold and Halpern-Leistner, Daniel and
              Xu, Chenyang},
     TITLE = {Reductivity of the automorphism group of {$K$}-polystable
              {F}ano varieties},
   JOURNAL = {Invent. Math.},
  FJOURNAL = {Inventiones Mathematicae},
    VOLUME = {222},
      YEAR = {2020},
    NUMBER = {3},
     PAGES = {995--1032},
      ISSN = {0020-9910,1432-1297},
   MRCLASS = {14J45 (14D23 14E30 14J10)},
  MRNUMBER = {4169054},
MRREVIEWER = {James\ McKernan},
       DOI = {10.1007/s00222-020-00987-2},
       URL = {https://doi.org/10.1007/s00222-020-00987-2},
}

@article {open_Kmoduli,
    AUTHOR = {Blum, Harold and Liu, Yuchen and Xu, Chenyang},
     TITLE = {Openness of {K}-semistability for {F}ano varieties},
   JOURNAL = {Duke Math. J.},
  FJOURNAL = {Duke Mathematical Journal},
    VOLUME = {171},
      YEAR = {2022},
    NUMBER = {13},
     PAGES = {2753--2797},
      ISSN = {0012-7094,1547-7398},
   MRCLASS = {14J45 (14E30)},
  MRNUMBER = {4505846},
MRREVIEWER = {Jihun\ Park},
       DOI = {10.1215/00127094-2022-0054},
       URL = {https://doi.org/10.1215/00127094-2022-0054},
}

@article {smoothable_Kmoduli,
    AUTHOR = {Li, Chi and Wang, Xiaowei and Xu, Chenyang},
     TITLE = {On the proper moduli spaces of smoothable
              {K}\"ahler-{E}instein {F}ano varieties},
   JOURNAL = {Duke Math. J.},
  FJOURNAL = {Duke Mathematical Journal},
    VOLUME = {168},
      YEAR = {2019},
    NUMBER = {8},
     PAGES = {1387--1459},
      ISSN = {0012-7094,1547-7398},
   MRCLASS = {14J45 (14D20 14J10 53C25 53C55)},
  MRNUMBER = {3959862},
MRREVIEWER = {Indranil\ Biswas},
       DOI = {10.1215/00127094-2018-0069},
       URL = {https://doi.org/10.1215/00127094-2018-0069},
}

@article {effective,
    AUTHOR = {Di Lorenzo, Andrea and Inchiostro, Giovanni},
     TITLE = {Effective morphisms and quotient stacks},
   JOURNAL = {Int. Math. Res. Not. IMRN},
  FJOURNAL = {International Mathematics Research Notices. IMRN},
      YEAR = {2024},
    NUMBER = {11},
     PAGES = {9091--9113},
      ISSN = {1073-7928,1687-0247},
   MRCLASS = {14A20 (14D23)},
  MRNUMBER = {4756104},
MRREVIEWER = {Jon\ Eivind\ Vatne},
       DOI = {10.1093/imrn/rnae010},
       URL = {https://doi.org/10.1093/imrn/rnae010},
}

@incollection {ACV,
    AUTHOR = {Abramovich, Dan and Corti, Alessio and Vistoli, Angelo},
     TITLE = {Twisted bundles and admissible covers},
      NOTE = {Special issue in honor of Steven L. Kleiman},
   JOURNAL = {Comm. Algebra},
  FJOURNAL = {Communications in Algebra},
    VOLUME = {31},
      YEAR = {2003},
    NUMBER = {8},
     PAGES = {3547--3618},
      ISSN = {0092-7872,1532-4125},
   MRCLASS = {14H10 (14A20 14H30)},
  MRNUMBER = {2007376},
MRREVIEWER = {Andrew\ Kresch},
       DOI = {10.1081/AGB-120022434},
       URL = {https://doi.org/10.1081/AGB-120022434},
}

@misc{warped,
      title={Beyond twisted maps: applications to motivic integration}, 
      author={Matthew Satriano and Jeremy Usatine},
      year={2024},
      eprint={2309.11434},
      archivePrefix={arXiv},
      primaryClass={math.AG},
      url={https://arxiv.org/abs/2309.11434}, 
}

@misc{mjjeon,
      title={Resolution of Indeterminacy of Rational Maps to Proper Tame Stacks}, 
      author={Myeong Jae Jeon},
      year={2025},
      eprint={2506.14969},
      archivePrefix={arXiv},
      primaryClass={math.AG},
      url={https://arxiv.org/abs/2506.14969}, 
}

@book {flipsandabundance,
     TITLE = {Flips and abundance for algebraic threefolds},
      NOTE = {Papers from the Second Summer Seminar on Algebraic Geometry
              held at the University of Utah, Salt Lake City, Utah, August
              1991,
              Ast\'erisque No. 211 (1992)},
 PUBLISHER = {Soci\'et\'e{} Math\'ematique de France, Paris},
      YEAR = {1992},
     PAGES = {1--258},
      ISSN = {0303-1179,2492-5926},
   MRCLASS = {14E30 (14E35 14M10)},
  MRNUMBER = {1225842},
MRREVIEWER = {Mark\ Gross},
}

@article {kawakita,
    AUTHOR = {Kawakita, Masayuki},
     TITLE = {Inversion of adjunction on log canonicity},
   JOURNAL = {Invent. Math.},
  FJOURNAL = {Inventiones Mathematicae},
    VOLUME = {167},
      YEAR = {2007},
    NUMBER = {1},
     PAGES = {129--133},
      ISSN = {0020-9910,1432-1297},
   MRCLASS = {14E30 (14N30)},
  MRNUMBER = {2264806},
MRREVIEWER = {Carla\ Novelli},
       DOI = {10.1007/s00222-006-0008-z},
       URL = {https://doi.org/10.1007/s00222-006-0008-z},
}

@article {odaka,
    AUTHOR = {Odaka, Yuji},
     TITLE = {The {GIT} stability of polarized varieties via discrepancy},
   JOURNAL = {Ann. of Math. (2)},
  FJOURNAL = {Annals of Mathematics. Second Series},
    VOLUME = {177},
      YEAR = {2013},
    NUMBER = {2},
     PAGES = {645--661},
      ISSN = {0003-486X,1939-8980},
   MRCLASS = {14C20 (14J17 14L24 32Q26)},
  MRNUMBER = {3010808},
MRREVIEWER = {Daniel\ Greb},
       DOI = {10.4007/annals.2013.177.2.6},
       URL = {https://doi.org/10.4007/annals.2013.177.2.6},
}

@article{chernousov2006resolving,
  title={Resolving G-torsors by abelian base extensions},
  author={Chernousov, Vladimir and Gille, Philippe and Reichstein, Zinovy},
  journal={Journal of Algebra},
  volume={296},
  number={2},
  pages={561--581},
  year={2006},
  publisher={Elsevier}
}

@article{chernousov2008reduction,
  title={Reduction of structure for torsors over semilocal rings},
  author={Chernousov, Vladimir and Gille, Philippe and Reichstein, Zinovy},
  journal={manuscripta mathematica},
  volume={126},
  pages={465--480},
  year={2008},
  publisher={Springer}
}

@article {MR3961332,
    AUTHOR = {Ascher, Kenneth and Bejleri, Dori},
     TITLE = {Moduli of fibered surface pairs from twisted stable maps},
   JOURNAL = {Math. Ann.},
  FJOURNAL = {Mathematische Annalen},
    VOLUME = {374},
      YEAR = {2019},
    NUMBER = {1-2},
     PAGES = {1007--1032},
      ISSN = {0025-5831,1432-1807},
   MRCLASS = {14J10 (14D23)},
  MRNUMBER = {3961332},
MRREVIEWER = {Izzet\ Coskun},
       DOI = {10.1007/s00208-018-1697-5},
       URL = {https://doi.org/10.1007/s00208-018-1697-5},
}

@misc{bejleri2024heightmodulicyclotomicstacks,
      title={Height moduli on cyclotomic stacks and counting elliptic curves over function fields}, 
      author={Dori Bejleri and Jun-Yong Park and Matthew Satriano},
      year={2024},
      eprint={2210.04450},
      archivePrefix={arXiv},
      primaryClass={math.NT},
      url={https://arxiv.org/abs/2210.04450}, 
}

@misc{gs_parahoric,
      title={Reductification of parahoric group schemes}, 
      author={Arnab Kundu},
      year={2026},
      eprint={2603.05554},
      archivePrefix={arXiv},
      primaryClass={math.AG},
      url={https://arxiv.org/abs/2603.05554}, 
}

@article {MR4703674,
    AUTHOR = {Damiolini, Chiara},
     TITLE = {On equivariant bundles and their moduli spaces},
   JOURNAL = {C. R. Math. Acad. Sci. Paris},
  FJOURNAL = {Comptes Rendus Math\'ematique. Acad\'emie des Sciences. Paris},
    VOLUME = {362},
      YEAR = {2024},
     PAGES = {55--62},
      ISSN = {1631-073X,1778-3569},
   MRCLASS = {14L15 (14H60 14L30)},
  MRNUMBER = {4703674},
MRREVIEWER = {Zinovy\ Reichstein},
       DOI = {10.5802/crmath.524},
       URL = {https://doi.org/10.5802/crmath.524},
}

@misc{properness_kmoduli_new,
    title = {Relative stability theory of properness of K-moduli}, 
    author = {Harold Blum and Yuchen Liu and Chenyang Xu and Ziquan Zhuang}, 
    year = {2025},
    howpublished = {In preparation},
}

@misc{good_moduli_gerbe,
    title = {A primer on good moduli space gerbes},
    author = {Dori Bejleri and Giovanni Inchiostro and Siddharth Mathur and Matthew Satriano}, 
    year = {2025},
    howpublished = {In preparation},
}

@misc{javier_hyperelliptic,
    title = {Classification of singular fibers of hyperelliptic fibrations},
    author = {Javier Reyes},
    year = {2025},
    howpublished = {In preparation}
}

@misc{ConradReductive,
  author       = {{Conrad}, Brian},
  title        = {Reductive Group Schemes},
  howpublished = {\url{http://math.stanford.edu/~conrad/papers/luminysga3.pdf}},
  Note = {Soci\'et\'e math\'ematique de France},
  year         = {2014},
}

@ARTICLE{DardaYasuda,
       author = {{Darda}, Ratko and {Yasuda}, Takehiko},
        title = "{The {B}atyrev-{M}anin conjecture for {DM} stacks}",
      journal = {arXiv e-prints},
         year = 2022,
        month = july,
          eid = {arXiv:2207.03645},
        pages = {arXiv:2207.03645},
archivePrefix = {arXiv},
       eprint = {2207.03645},
 primaryClass = {math.AG},
}

@article {ER,
    AUTHOR = {Edidin, Dan and Rydh, David},
     TITLE = {Canonical reduction of stabilizers for {A}rtin stacks with
              good moduli spaces},
   JOURNAL = {Duke Math. J.},
  FJOURNAL = {Duke Mathematical Journal},
    VOLUME = {170},
      YEAR = {2021},
    NUMBER = {5},
     PAGES = {827--880},
      ISSN = {0012-7094,1547-7398},
   MRCLASS = {14D23 (14E15 14L24)},
  MRNUMBER = {4255045},
MRREVIEWER = {Ariyan\ Javanpeykar},
       DOI = {10.1215/00127094-2020-0050},
       URL = {https://doi.org/10.1215/00127094-2020-0050},
}

@article {ESZB,
    AUTHOR = {Ellenberg, Jordan S. and Satriano, Matthew and Zureick-Brown,
              David},
     TITLE = {Heights on stacks and a generalized {B}atyrev-{M}anin-{M}alle
              conjecture},
   JOURNAL = {Forum Math. Sigma},
  FJOURNAL = {Forum of Mathematics. Sigma},
    VOLUME = {11},
      YEAR = {2023},
     PAGES = {Paper No. e14, 54},
   MRCLASS = {11G50 (11G35 14D23 14G05 14G25 14G40)},
  MRNUMBER = {4557890},
       DOI = {10.1017/fms.2023.5},
       URL = {https://doi.org/10.1017/fms.2023.5},
}

@article{alexeev2012non,
  title={Non-normal abelian covers},
  author={Alexeev, Valery and Pardini, Rita},
  journal={Compositio Mathematica},
  volume={148},
  number={4},
  pages={1051--1084},
  year={2012},
  publisher={London Mathematical Society}
}

@article{liu2022finite,
  title={Finite generation for valuations computing stability thresholds and applications to K-stability},
  author={Liu, Yuchen and Xu, Chenyang and Zhuang, Ziquan},
  journal={Annals of Mathematics},
  volume={196},
  number={2},
  pages={507--566},
  year={2022},
  publisher={Department of Mathematics, Princeton University Princeton, New Jersey, USA}
}

@article{di2022degenerations,
  title={Degenerations of twisted maps to algebraic stacks},
  author={Di Lorenzo, Andrea and Inchiostro, Giovanni},
  journal={arXiv preprint arXiv:2210.03806},
  year={2022}
}

@misc{SGA3III,
  title={S{\'e}minaire de g{\'e}om{\'e}trie alg{\'e}brique du Bois Marie 1962-64. Sch{\'e}mas en groupes ({SGA} 3). {T}ome {III}: Structures des sch{\'e}mas en groupes reductiv},
  author={Demazure, M and Grothendieck, A and Artin, M and Bertin, JE and Gabriel, P and Raynaud, M and Serre, JP},
  year={2011},
  publisher={volume}
}

@article {SGA3,
    AUTHOR = {Milne, J. S.},

     TITLE = {{\it {S}\'{e}minaire de g\'{e}om\'{e}trie alg\'{e}brique du
              {B}ois {M}arie 1962--64. {S}ch\'{e}mas en groupes ({SGA} 3)}.
              {T}om {I}, {T}om {III} [book review of MR2867621; MR2867622]},
   JOURNAL = {Eur. Math. Soc. Newsl.},
  FJOURNAL = {European Mathematical Society. Newsletter},
    NUMBER = {87},
      YEAR = {2013},
     PAGES = {50--52},
      ISSN = {1027-488X},
   MRCLASS = {00A17},
  MRNUMBER = {3075444},
LABEL={SGA3},
KEY= {SGA3},
}

@book{QingLiu2002algebraic,
  title={Algebraic geometry and arithmetic curves},
  author={Liu, Qing},
  volume={6},
  year={2002},
  publisher={OUP Oxford}
}

@article{nagata1961complete,
  title={Complete reducibility of rational representations of a matric group},
  author={Nagata, Masayoshi},
  journal={Journal of mathematics of Kyoto University},
  volume={1},
  number={1},
  pages={87--99},
  year={1961},
  publisher={Duke University Press}
}

@misc{stacks-project,
  author       = {The {Stacks project authors}},
  title        = {The Stacks project},
  howpublished = {\url{https://stacks.math.columbia.edu}},
  year         = {2021},
}

@article{bejleri2024proper,
  title={Proper splittings and projectivity for good moduli spaces},
  author={Bejleri, Dori and Elmanto, Elden and Satriano, Matthew},
  journal={arXiv preprint arXiv:2408.11057},
  year={2024}
}
\end{document}